\documentclass[11pt,a4paper]{article}
\usepackage[utf8]{inputenc}
\usepackage[english]{babel}
\usepackage[T1]{fontenc} % -> Im PDF kann man Sonderzeichen wie bei Tur\'an kopieren
\usepackage{lmodern} %sinnvoll, Schrift nicht mehr pixlig
\usepackage{xcolor} %Paket fuer Farbe
\usepackage{color}
\usepackage{amssymb} %fuer natuerliche-zahlen-symbol
\usepackage{mathrsfs} %fuer geschwungenes L
\usepackage{amsmath}
\usepackage{amsthm}%fuer Theoreme, proof
\usepackage{amsopn}
\usepackage{enumitem}%for enumerate/itemize without indentation -> [wide=0pt] or [leftmargin=*] -> for [label=\textbf{S.\arabic*}]
\usepackage{graphicx} % für \bigcupdot notwendig
\usepackage{subcaption}
\usepackage{xspace} %bei Abkuerzung von d.h. \xspace in newcommand
\usepackage{mathtools} %noetig fuer declare paired delimiter -> gaussklammer so definiert
\usepackage{scalerel,stackengine} % for the newcommand of the of fourier hat
\usepackage{url}
\usepackage{hyperref}
\usepackage{doi}
\usepackage{tikz}
\usepackage{tikz-cd} %commuting diagram
\usepackage{nomencl} % list of notation
\usepackage[margin=3cm]{geometry}
\usepackage{pictexwd,dcpic}
\usepackage[titletoc]{appendix}

\theoremstyle{plain}

\newtheorem{Theorem}{Theorem}[section]
\newtheorem{Lemma}[Theorem]{Lemma}
\newtheorem{Corollary}[Theorem]{Corollary}
\newtheorem{Proposition}[Theorem]{Proposition}

\theoremstyle{definition}

\newtheorem{Definition}[Theorem]{Definition}
\newtheorem{Example}[Theorem]{Example}

\theoremstyle{remark}

\newtheorem{Remark}[Theorem]{Remark}

\makeatletter
\newtheoremstyle{case}{3pt}{3pt}{\addtolength{\@totalleftmargin}{2.0em}\addtolength{\linewidth}{-1.5em}\parshape 1 1.5em \linewidth}{}{\kern-1.5em\itshape}{.}{ }{\thmnote{#3}}
\makeatother
\theoremstyle{case}

\makeatletter
\newtheoremstyle{caseinside}{3pt}{3pt}{\addtolength{\@totalleftmargin}{4.0em}\addtolength{\linewidth}{-3.0em}\parshape 1 3.0em \linewidth}{}{\kern-1.5em\itshape}{.}{ }{\thmnote{#3}}
\makeatother
\theoremstyle{caseinside}

\DeclareMathOperator{\Aut}{Aut} %automorphism group

\DeclareMathOperator{\dist}{dist}
\DeclareMathOperator{\E}{E}     %Euclidean group
\DeclareMathOperator{\GL}{\mathrm{GL}}
\DeclareMathOperator{\Ind}{Ind} %induced representation

\DeclareMathOperator{\Trans}{Trans} % group of translations \trans(d)

\DeclareMathOperator{\U}{\mathrm{U}} % unitary group
\renewcommand{\O}{\operatorname{{\mathrm O}}} %orthogonal group
\DeclarePairedDelimiter{\abs}{\lvert}{\rvert}
\DeclarePairedDelimiter{\angles}{\langle}{\rangle}
\DeclarePairedDelimiter{\norm}{\lVert}{\rVert}
\DeclarePairedDelimiter{\parens}{\lparen}{\rparen}

\DeclarePairedDelimiter{\braces}{\lbrace}{\rbrace}

\DeclarePairedDelimiterX{\iso}[2]{(}{)}{#1,#2} % (A,b) is better than (A|b) since E(d)=O(d)\ltimes\R^d, but for examples like (0&0\\1&1|0\\1) is (A|b) better than (A,b)
\DeclarePairedDelimiterX{\set}[2]{\{}{\}}{#1\,\delimsize|\,#2}
\DeclarePairedDelimiterX{\scalar}[2]{\langle}{\rangle}{#1,#2}
\makeatletter
\newcommand{\@absonestar}[1]{\abs*{#1}_1}
\newcommand{\@absonenostar}[2][]{\abs[#1]{#2}_1}
\newcommand{\absone}{\@ifstar\@absonestar\@absonenostar}
\makeatother
\makeatletter
\newcommand{\@zeronormstar}[2]{\norm*{#1}_{#2,0}}
\newcommand{\@zeronormnostar}[3][]{\norm[#1]{#2}_{#3,0}}
\newcommand{\zeronorm}{\@ifstar\@zeronormstar\@zeronormnostar}
\makeatother
\makeatletter
\newcommand{\@nablanormstar}[2]{\norm*{#1}_{#2,\nabla}}
\newcommand{\@nablanormnostar}[3][]{\norm[#1]{#2}_{#3,\nabla}}
\newcommand{\nablanorm}{\@ifstar\@nablanormstar\@nablanormnostar}
\makeatother
\makeatletter
\newcommand{\@nablazeronormstar}[2]{\norm*{#1}_{#2,0,\nabla,0}}
\newcommand{\@nablazeronormnostar}[3][]{\norm[#1]{#2}_{#3,\nabla,0}}
\newcommand{\nablazeronorm}{\@ifstar\@nablazeronormstar\@nablazeronormnostar}
\makeatother
\makeatletter
\newcommand{\@newnormstar}[2]{\norm*{#1}_{#2,0,0}}
\newcommand{\@newnormnostar}[3][]{\norm[#1]{#2}_{#3,0,0}}
\newcommand{\newnorm}{\@ifstar\@newnormstar\@newnormnostar}
\makeatother
\makeatletter
\newcommand{\@newnablanormstar}[2]{\norm*{#1}_{#2,\nabla,0,0}}
\newcommand{\@newnablanormnostar}[3][]{\norm[#1]{#2}_{#3,\nabla,0,0}}
\newcommand{\newnablanorm}{\@ifstar\@newnablanormstar\@newnablanormnostar}
\makeatother
\makeatletter
\providecommand*{\cupdot}{%
  \mathbin{%
    \mathpalette\@cupdot{}%
  }%
}
\newcommand*{\@cupdot}[2]{%
  \ooalign{%
    $\m@th#1\cup$\cr
    \sbox0{$#1\cup$}%
    \dimen@=\ht0 %
    \sbox0{$\m@th#1\cdot$}%
    \advance\dimen@ by -\ht0 %
    \dimen@=.5\dimen@
    \hidewidth\raise\dimen@\box0\hidewidth
  }%
}
\providecommand*{\bigcupdot}{%
  \mathop{%
    \vphantom{\bigcup}%
    \mathpalette\@bigcupdot{}%
  }%
}
\newcommand*{\@bigcupdot}[2]{%
  \ooalign{%
    $\m@th#1\bigcup$\cr
    \sbox0{$#1\bigcup$}%
    \dimen@=\ht0 %
    \advance\dimen@ by -\dp0 %
    \sbox0{\scalebox{2}{$\m@th#1\cdot$}}%
    \advance\dimen@ by -\ht0 %
    \dimen@=.5\dimen@
    \hidewidth\raise\dimen@\box0\hidewidth
  }%
}
\makeatother

\stackMath
\newcommand\fourier[1]{%
\savestack{\tmpbox}{\stretchto{%
  \scaleto{%
    \scalerel*[\widthof{\ensuremath{#1}}]{\kern-.6pt\bigwedge\kern-.6pt}%
    {\rule[-\textheight/2]{1ex}{\textheight}}%WIDTH-LIMITED BIG WEDGE
  }{\textheight}% 
}{0.5ex}}%
\stackon[1pt]{#1}{\tmpbox}%
}
\par

\newcommand{\conj}[1]{\overline{#1}}
\newcommand{\dual}[1]{\widehat{#1}}
\newcommand{\fdot}{{{}\cdot{}}} %spacing of $f(\cdot)$ in not equal to that of $f({}\cdot{})$, also $\lvert\cdot\rvert$ vs $\lvert{}\cdot{}\rvert$ -> good for norm and abs
\newcommand{\gdot}{\cdot} %group action
\newcommand{\gplus}[2]{#1\oplus #2}

\newcommand{\simrg}{\sim} % {\sim_{\R^{d_2},\G}}

\newcommand{\euler}{\mathrm e}
\newcommand{\iu}{\mathrm i}
\newcommand{\C}{{\mathbb C}}
\newcommand{\N}{{\mathbb N}}
\newcommand{\Q}{{\mathbb Q}}
\newcommand{\R}{{\mathbb R}}
\newcommand{\Z}{{\mathbb Z}}
\renewcommand{\epsilon}{\varepsilon} %Bernd benutzt varepsilon
\renewcommand{\phi}{\varphi} %Bernd benutzt varphi (cut-off function)

\newcommand{\F}{\mathcal F} %'big' finite group
\newcommand{\G}{\mathcal G}

\renewcommand{\H}{\mathcal H}
\newcommand{\id}{{id}} %identity of a group
\newcommand{\K}{\mathcal K} % subgroup of \H
\newcommand{\LL}{L_\SG} % lattice
\newcommand{\dualL}{L_\SG^*} % dual lattice
\newcommand{\MM}{M_0} % M_0=\set{n\in\N}{\T^m is a normal subgroup of \G}
\renewcommand{\P}{\mathcal P} %'big' point set
\newcommand{\CC}{\mathcal C} % Cuboid
\newcommand{\EE}{\mathcal E} % irreducble representations
\newcommand{\NN}{\mathcal N} % normal subgroup and nearest neighbors
\newcommand{\RR}{\mathcal R} % range (derivative)
\newcommand{\SG}{\mathcal S} %space group in R^{d_1}
\newcommand{\T}{\mathcal T} %'big' translation group
\newcommand{\rot}{L} % linear component
\newcommand{\trans}{\tau} % translational component
\newcommand{\bartau}{\bar\tau} % \taubar\colon \P_0\times\P_0 \to \R^d
\newcommand{\tildetau}{\tilde\tau} % % \tildetau\colon\P_0\to\R^d

\newcommand{\Per}{L_\mathrm{per}^\infty}
\newcommand{\eg}{\mbox{e.\,g.}\xspace}
\newcommand{\ie}{\mbox{i.\,e.}\xspace}
\begin{document}
\begin{center}
\begin{Large}
On discrete groups of Euclidean isometries: representation \\ theory, harmonic analysis and splitting properties 
\end{Large}
\\[0.5cm]
\begin{large}
Bernd Schmidt\footnote{Universität Augsburg, Germany, {\tt bernd.schmidt@math.uni-augsburg.de}} and 
Martin Steinbach\footnote{Universität Augsburg, Germany, {\tt steinbachmartin@gmx.de}} 
\end{large}
\\[0.5cm]
\today
\\[1cm]
\end{center}

\begin{abstract}
We study structural properties and the harmonic analysis of discrete subgroups of the Euclidean group. In particular, we 1.\ obtain an efficient description of their dual space, 2.\ develop Fourier analysis methods for periodic mappings on them, and 3.\ prove a Schur-Zassenhaus type splitting result. 
\end{abstract}

\medskip 


\tableofcontents

%-------------------------------------------------------------------
%-------------------------------------------------------------------
\section{Introduction}

The main objective of the present contribution is a study of various aspects of discrete subgroups of the Euclidean group $\E(d)$ concerning their representations, harmonic analysis and splitting properties. Special cases of such groups are finite subgroups of the rotation group $\O(d)$ and the crystallographic space groups. While these objects are very well studied, cf.\ \eg \cite{DuVal1964,Bradley2010,Brown1978} for classic references, the theory for general discrete groups of Euclidean isometries is considerably less developed. In a sense that is made precise below, they interpolate and combine aspects of finite rotation groups and space groups as they might be assumed to embed into a subgroup of $\gplus{\O(d_1)}{\mathcal{S}}$ for some $d_2$-dimensional space group $\mathcal{S}$, where $d_1+d_2=d$. We will see that it is possible to identify a ``translational part'' of such a group featuring some properties analogous to those of a $d_2$-dimensional lattice. This will allow us to 1.\ obtain an efficient description of the dual space of discrete subgroups of $\E(d)$, 2.\ to develop Fourier analysis methods for periodic mappings on them, and 3.\ to prove a Schur-Zassenhaus type splitting result with respect to the translational part. 

Besides its intrinsic mathematical interest, our motivation originates in a physical question on the stability of so-called {\em Objective Structures}. These particle systems were introduced by James in \cite{James2006} as a far reaching generalization of lattice systems and have been deployed successfully to describe a remarkable number of important structures ranging from biology (parts of viruses) to nanoscience (carbon nanotubes). 

In order to motivate and illustrate our investigations let us consider $\Z^d$ as the most basic and classical example of a lattice system. Localized mappings $u \in \ell^1(\Z^d)$ (say) can be conveniently analyzed in terms of their Fourier transform which in particular induces a resolution into plane waves $\euler^{2\pi\iu \scalar k\cdot}$ with dominant contributions from small ``wave numbers'' $|k|$. So as to account for lattice mappings with significant contributions from large large wave numbers, a feasible strategy is to investigate $N$-periodic functions $u : \Z^d \to \C$, $N \in \N$, for which $u(x + N e_i) = u(x)$ for all $i \in \{1, \ldots, d\}$ in the asymptotic regime $N \to \infty$. Such a function might alternatively be looked at as a function on $\Z_N^d$ and is described by its discrete Fourier transform 
$$\fourier u(k) 
   = N^{-d} \sum_{x \in \Z_N^d} u(x) \, \euler^{-2\pi\iu \scalar kx}, $$ 
$k \in \{0, \frac{1}{N}, \ldots, \frac{N-1}{N}\}^d$, as 
$$u(x) 
   = \sum_{k \in N^{-1} \Z_N^d} \fourier u(k) \, \euler^{2\pi\iu\scalar kx}. $$ 
Here the ``wave vectors'' $k$ are eventually dense in the unit cell $[0,1]^d$ as the period $N$ becomes larger and larger. In this way one obtains a Fourier descriptions for all periodic functions which, upon truncating the maximal periodicity at finite $N$, lends itself to a directly controllable approximation in numerical simulations.

By a simple coordinate change such an analysis applies to general lattice systems $A \Z^d$, $A \in \GL(d, \R)$, which are easily seen to be the orbit of a single point under the action of a discrete group of translations on $\R^d$. These point sets are of paramount interest in solid state physics where they might describe positions of atoms and molecules. More generally, an Objective Structure can be defined as the orbit of a single point under the action of a general discrete subgroup of the Euclidean group $\E(d)$ on $\R^d$, cf.\ \cite{James2006}. Those structures are thus characterized by the fact that any two points ``see'' an identical environment of other points, modulo a rigid motion. 

One of the aims of the present contribution is to develop an extension of the above described Fourier analysis for discrete translation groups to general discrete groups of Euclidean isometries. While in principle a Fourier transform is defined on the dual space of such a group, an efficient description of these spaces appears to be missing. Moreover, the incorporation of periodic mappings with significant ``long wave-length'' contributions is non-trivial. In fact, due to a possible lack of periodicity, even the definition of a quantity that can be interpreted as a wave-length is not obvious. Our main goal is therefore, by exploiting the special structure of discrete subgroups of $\E(d)$, to provide an efficient and extensive description of their dual spaces. In particular we identify a finite union of convex ``wave vector domains'' for each such space, which reflects the existence of an underlaying translational part of finite index.

We proceed to give a more detailed account of our results and a plan of the paper. In Section~\ref{section:structure} we first collect basic definitions and properties of the Euclidean group and space groups and cite a characterization of discrete subgroups of the Euclidean group: Up to conjugation in $\E(d)$ a discrete subgroup $\G$ of $\E(d)$ embeds into a subgroup of $\O(d_1) \oplus \mathcal{S}$ for a spacegroup $\mathcal{S} \subset E(d_2)$, where $d_1+d_2=d$, with surjective projection $\pi$ onto $\mathcal{S}$, cf.\ Fig.~\ref{fig:group-structure}.  
%
%\begin{center}
\begin{figure}[h!]
\null\hspace{3.3cm}
\begindc{\commdiag}[300]
\obj(0,2)[aa]{$ \G $}
\obj(3,2)[bb]{$ \O(d_1) \oplus \mathcal{S} $}
\obj(6,2)[cc]{$ \O(d_1) \oplus \E(d_2) $} 
\obj(0,0)[dd]{$ \pi(\G) $}
\obj(3,0)[ee]{$ \mathcal{S} $}
\obj(6,0)[ff]{$ \E(d_2) $}
\mor{aa}{bb}{}[\atright,6] % ,1 gibt gestrichelte Linie
\mor{bb}{cc}{}[\atright,6]
\mor{aa}{dd}{$ \pi $}[\atleft,0]
\mor{bb}{ee}{$ \pi $}[\atleft,0]
\mor{cc}{ff}{$ \pi $}[\atleft,0]
\mor{dd}{ee}{}[\atleft, \equalline]
\mor{ee}{ff}{}[\atleft, 6]
\enddc
\caption{\label{fig:group-structure} Structure of a discrete subgroup of $\E(d)$.}
\end{figure}
%\end{center}

Whereas the spacegroup $\mathcal{S}$ has a rich translation group $\T_{\mathcal{S}}$, this is in general not the case for $\G$. To overcome this problem, in Section~\ref{section:sections} we fix a section $\T \subset \G$ of $\T_{\mathcal{S}}$, which, however, will not be a group in general. We also set $\F = \ker(\pi) \cap \G$ and analyze $\T$ and $\F$ in some detail. A first remarkable fact then is that $\T^N$ (set of $N$-fold products of elements in $\T$) for suitable $N \in \N$ is a subgroup and even a normal divisor, see Theorem~\ref{Theorem:Ring}, which is essential in our later considerations on periodic functions. 

In Section~\ref{subsection:dual} we obtain our main result on the structure of representations of $\G$. We consider the subgroup $\T\F$ of $\G$, which corresponds to the set of preimages of $\T_{\mathcal{S}}$ under $\pi$, and analyze in detail the representations of $\T\F$ and the induced representations of $\T\F$ in $\G$. In Theorem~\ref{Theorem:RepSet} we prove that the latter decompose into a finite disjoint union $\bigsqcup_{\rho\in R}\R^{d_2}/\G_\rho$ of orbit spaces (or fundamental domains) of certain space groups $\G_\rho$ on $\R^{d_2}$ via 
\[ (\G_\rho\cdot k,\rho)
   \mapsto\Ind_{\T\F}^\G(\euler^{-2\pi\iu \scalar k\cdot}\rho), \] 
for a suitable finite $R \subset \dual{\T\F}$. This in particular allows for an interpretation of $\bigsqcup_{\rho\in R}\R^{d_2}/\G_\rho$ as a set of ``wave vectors''. A version for periodic representations is given in Theorem~\ref{Theorem:RepSys} where the orbit spaces $\R^{d_2}/\G_\rho$ are replaced by suitable rescalings of $\dualL$, the dual lattice of translations in $\mathcal{S}$. As a result of Theorem~\ref{Theorem:RepSet}, we obtain that -- up to a negligible set -- the whole dual space $\dual\G$ is equal to the same set of induced representations. More precisely, for each $\rho \in R$ there is a zero-set $N_\rho\subset\R^{d_2}/\G_\rho$ (with respect to the Lebesgue-measure) and a zero-set $N\subset\dual\G$ (with respect to the Plancherel-measure on $\dual\G$) such that the above mapping is bijective when restricted to 
\[ \bigsqcup_{\rho\in R}(\R^{d_2}/\G_\rho)\setminus N_\rho\to\dual\G\setminus N, \] 
cf.\ Theorem~\ref{Theorem:Bijektion}. We remark that, as compared to general results in this direction obtained with the Mackey machine (see in particular \cite{Mackey:58,KleppnerLipsman:72} and also cp.\ Theorem~\ref{Theorem:Mackey-machine} below), we obtain an explicit labeling of representations with ``wave vectors'' in a finite union of convex domains in $\R^{d_2}$, which in turn are identified as fundamental domains of associated space groups.

In the following, comparatively elementary Section~\ref{subsection:fourier} we introduce an inner product space of functions that satisfy a suitable periodicity assumption. We then proceed to develop a harmonic analysis on such objects by defining the Fourier transform for both periodic and absolutely summable functions and formulating well-known theorems like the Plancherel formula within our setting.

In Section~\ref{Subsection:SemidirectProduct} we then address the question if $\G$ splits into a translational part (more precisely, some $\T^m \triangleleft \G$) and a finite complement. In general this is not the case. However, our main structural splitting result  \ref{Theorem:semidirectproductG_NNEW} provides such representations as   semidirect products for quotient groups of $\G$ with respect to a series of eventually sparse normal subgroups: If in addition $n\in \N$ is coprime to $m$ and $\abs{\SG/\T_\SG}$, then there is a group $\H\le\G/\T^{nm}$ such that 
\[ \G/\T^{nm}
   = \T^m/\T^{nm}\rtimes \H, \] 
where $\T^m/\T^{nm} \cong \Z_n^{d_2}$. (Note that for space groups and $m = 1$ such a result is mentioned in \cite{Bacry1988,Florek1993}.) 

Finally, Section~\ref{subsection:harmonic} in the appendix collects some well-known definitions and theorems of harmonic analysis like the definition and basic properties of dual spaces and induced representations.

We close this introduction with an outlook to applications on the stability analysis of particle systems. Assuming that particles at different sites interact, one is naturally led to the question if an Objective Structure corresponds to a stable configuration. Similar questions are by now well understood in lattice systems, see, \eg, \cite{Hudson2011}: At equilibrium configurations the second order Taylor approximation of the configurational energy is conveniently analyzed in Fourier space and formulae for stability constants under rather generic interaction assumptions are available \cite{Braun2016}. The results in the current contribution will indeed allow for an analogous characterization of stability constants for Objective Structures. This will be realized in the forthcoming contributions \cite{SchmidtSteinbach:21b,SchmidtSteinbach:21c}, where we provide a characterization that even leads to a numerical algorithm for testing the stability of a given structure and to novel applications to nanotubes.

%-------------------------------------------------------------------
\subsection*{Notation}
We will use the following notation. For all groups $G$ and subsets $S_1,S_2\subset G$ we denote
\[S_1S_2:=\set{s_1s_2}{s_1\in S_1,s_2\in S_2}\subset G\]
the product of group subsets.
For all $S\subset G$, $n\in\Z$ and $g\in G$ we denote
\begin{align*}
S^n&:=\set{s^n}{s\in S}\subset G
\shortintertext{and}
gS&:=\set{gs}{s\in S}\subset G.
\end{align*}
For two groups $G,H$ we write $H<G$ if $H$ is a proper subgroup of $G$ and $H\triangleleft G$ if $H$ is a normal subgroup of $G$.
For a subset $S$ of a group $G$ we write $\angles S$ for the subgroup generated by $S$.

Moreover, let $\N$ be the set of all positive integers $\{1,2,\dots\}$, $\Z_n$ be the group $\Z/(n\Z)$, $e_i$ be the $i^{\text{th}}$ standard coordinate vector $(0,\dots,0,1,\dots,0)\in\R^d$ and $I_n\in\R^{n\times n}$ be the identity matrix of size $n$. We use capital letters for matrices. For $A=(a_{ij})\in\C^{m\times n}$ and $B=(b_{ij})\in\C^{p\times q}$, their direct sum and their Kronecker product are 
\[A\oplus B:=\parens[\bigg]{\begin{matrix}A&0\\0&B\end{matrix}} \in \C^{(m+p)\times (n+q)}, \qquad 
A\otimes B:=\begin{bmatrix}
a_{11}B & \cdots & a_{1n}B\\
\vdots & \ddots & \vdots\\
a_{m1}B & \cdots & a_{mn}B
\end{bmatrix}\in \C^{mp\times nq}, 
\]
respectively. 
The Hermitian adjoint of $A$ is denoted $A^H$. 

%-------------------------------------------------------------------
\subsection*{Acknowledgements}
This work was partially supported by project 285722765 of the Deutsche
Forschungsgemeinschaft (DFG, German Research Foundation). 

%-------------------------------------------------------------------
%-------------------------------------------------------------------
\section{Discrete groups of Euclidean isometries}\label{section:structure}

This preliminary section serves to collect some basics on the Euclidean group acting on $\R^d$ and of its discrete subgroups. We also introduce some general notation. 

%-------------------------------------------------------------------
\subsection*{The Euclidean group}%\label{section:euclidean}
Let $d\in\N$ be the dimension.
We denote the set of all Euclidean distance preserving transformations of $\R^d$ into itself by the \emph{Euclidean group} $\E(d)$.
The elements of $\E(d)$ are called \emph{Euclidean isometries}.
It is well-known that the Euclidean group $\E(d)$ can be described concretely as the outer semidirect product of $\R^d$ and $\O(d)$, the orthogonal group in dimension $d$:
\[\E(d)=\O(d)\ltimes \R^d.\]
The group operation is given by
\[\iso{A_1}{b_1}\iso{A_2}{b_2}=\iso{A_1A_2}{b_1+A_1b_2}\]
for all $\iso{A_1}{b_1},\iso{A_1}{b_2}\in\E(d)$, and the inverse of $\iso{A}{b}\in\E(d)$ is
\[\iso{A}{b}^{-1}=\iso{A^{-1}}{-A^{-1}b}.\]
Moreover, we define the maps
\begin{align*}
\rot\colon\E(d)\to\O(d), \quad 
&\iso Ab\mapsto A \qquad \text{and} \\  
\trans\colon\E(d)\to\R^d, \quad 
&\iso Ab\mapsto b
\end{align*}
and for all $\iso Ab\in\E(d)$ we call $\rot(\iso Ab)$ the \emph{linear component} and $\trans(\iso Ab)$ the \emph{translation component} of $\iso Ab$.
Thus, 
\[g=\iso{I_d}{\trans(g)}\iso{\rot(g)}0\]
for every $g\in\E(d)$. 
We call an Euclidean isometry $\iso Ab$ a \emph{translation} if $A=I_d$.
All translations form the \emph{group of translations} $\Trans(d)$, which is the abelian subgroup of $\E(d)$ given by
\[\Trans(d):=\{I_d\}\ltimes\R^d.\]
We call a set of translations \emph{linearly independent} if their translation components are linearly independent. The natural group action of $\E(d)$ on $\R^d$ is given by
\[\iso Ab\gdot x:=Ax+b\qquad\text{for all }\iso Ab\in \E(d)\text{ and }x\in\R^d.\]
In this contribution we use a calligraphic font for subsets and particularly for subgroups of $\E(d)$.
For every group $\G<\E(d)$ we denote the \emph{orbit} of a point $x\in\R^d$ under the action of the group $\G$ by
\[\G\gdot x:=\set{g\gdot x}{g\in\G}.\]
We endow $\E(d)$ with the subspace topology of the Euclidean space $\R^{d\times d}\times\R^d$ such that $\E(d)$ is a topological group.
It is well-known that a subgroup $\G<\E(d)$ is discrete if and only if for every $x\in\R^d$ the orbit $\G\gdot x$ is discrete, see, \eg, \cite[Exercise I.1.4]{Charlap1986}.
In particular, every finite subgroup of $\E(d)$ is discrete.

A discrete group $\G<\E(d)$ is said to be \emph{decomposable} if the group representation
\begin{align*}
&\G\to\GL(d+1,\C)\\
&(A,b)\mapsto\begin{pmatrix}A&b\\0&1\end{pmatrix}
\end{align*}
is decomposable, \ie, there is a decomposition of $\R^{n+1}$ into the direct sum of two proper subspaces invariant under $\set{\parens{\begin{smallmatrix}A&b\\0&1\end{smallmatrix}}}{\iso Ab\in\G}$. 
% By Maschke's Theorem, since we are working over \C, we can thus treat (in)decomposability as equivalent to (ir)reducibility 
If this is not the case, the discrete group $\G$ is called \emph{indecomposable}, see, \eg, \cite[Appendix A.3]{Brown1978}.
An indecomposable discrete group $\G<\E(d)$ is also called a \emph{($d$-dimensional) space group}.
Below we also present a (well-known) characterization of the space groups and the decomposable discrete subgroups of $\E(d)$, respectively, which does not use representations.

In the physically important case $d=3$, all space groups and discrete decomposable subgroups of $\E(3)$ are well-known and classified, see, \eg, \cite{Aroyo2016} and \cite{Opechowski1986}, respectively. %Opechowski -> Section 7.4 and Chapter 11
%

%-------------------------------------------------------------------
\subsection*{Space groups}%\label{section:SpaceGroups}
The following theorem is well-known, see, \eg, \cite[Appendix A.3]{Brown1978}.
\begin{Theorem}\label{Theorem:spacegroup}
Let $d\in\N$ be the dimension and $\G < \E(d)$ a discrete subgroup. The following are equivalent: 
\begin{enumerate}
\item $\G$ is a space group. 
\item $\G$ contains $d$ linearly independent translations. 
\item The subgroup of translations of $\G$ is generated by $d$ linearly independent translations.
\end{enumerate}
\end{Theorem}
Also the following theorem is well-known.
\begin{Theorem}\label{Theorem:SpaceGroup}
Let $\G$ be a $d$-dimensional space group and $\T$ its subgroup of translations.
Then it holds:
\begin{enumerate}
\item\label{item:aaaa} The group $\T$ is a normal subgroup of $\G$ and isomorphic to $\Z^d$.
\item\label{item:bbbb} The \emph{point group} $\rot(\G)$ of $\G$ is finite.
\item\label{item:cccc} The map
\[\G/\T\to\rot(\G),\quad\iso Aa\T\mapsto A\]
is bijective and particularly, also $\G/\T$ is finite.
\end{enumerate}
\end{Theorem}
\begin{proof}
\ref{item:aaaa} This is clear by Theorem~\ref{Theorem:spacegroup}.
\ref{item:bbbb} See, \eg, \cite[Theorem I.3.1]{Charlap1986}.
\ref{item:cccc} It is easy to see that the map is bijective and by \ref{item:bbbb} the set $\G/\T$ is finite.
\end{proof}
\begin{Corollary}
Let $\G$ be a $d$-dimensional space group and $\T$ its subgroup of translations.
Then for all $N\in\N$ the set $\T^N$ is a normal subgroup of $\G$ and isomorphic to $\Z^d$.
\end{Corollary}
\begin{proof}
This is clear by Theorem~\ref{Theorem:SpaceGroup}\ref{item:aaaa}.
\end{proof}
%

%-------------------------------------------------------------------
\subsection*{Discrete subgroups of the Euclidean group}%\label{section:DecomposableDiscreteSubgroups}
We recall that two subgroups $\G_1,\G_2<\E(d)$ are termed \emph{conjugate} subgroups under the group $\E(d)$ if there exists some $g\in\E(d)$ such that $g^{-1}\G_1 g=\G_2$.
Note that every such conjugation corresponds to a coordinate transformation in $\R^d$.

Now we may state the following well-known characterization of the discrete subgroups of $\E(d)$.
For this purpose for all $d_1,d_2\in\N$ we define the group homomorphism
\begin{align*}
\gplus{}{}\colon&\O(d_1)\times\E(d_2)\to\E(d_1+d_2)\\
&(A_1,\iso{A_2}{b_2})\mapsto \gplus{A_1}{\iso{A_2}{b_2}}:=\iso[\bigg]{\parens[\bigg]{\begin{matrix}A_1&0\\0&A_2\end{matrix}}}{\parens[\bigg]{\begin{matrix}0\\b_2\end{matrix}}}.
\end{align*}
\begin{Theorem}\label{Theorem:Browndecomposablediscretegroup}
Let $d\in\N$ be the dimension and $\G<\E(d)$ be discrete.
Then there exist $d_1,d_2\in\N_0$ such that $d=d_1+d_2$, a $d_2$-dimensional space group $\SG$ and a discrete group $\G'<\gplus{\O(d_1)}{\SG}$ such that $\G$ is conjugate under $\E(d)$ to $\G'$ and $\pi(\G')=\SG$, where $\pi$ is the natural surjective homomorphism $\gplus{\O(d_1)}{\E(d_2)}\to\E(d_2)$, $\gplus Ag\mapsto g$.
\end{Theorem}
\begin{proof}
Let $d\in\N$ be the dimension and $\G<\E(d)$ be discrete.
If $\G$ is a space group, the assertion is trivial.
If $\G$ is finite, then $\G$ is conjugate under $\E(d)$ to a finite subgroup of $\O(d)\ltimes \{0_d\}\cong\O(d)$, see, \eg, \cite[Section 4.12]{Opechowski1986}.
If $\G$ is an infinite decomposable discrete subgroup of $\E(d)$, the assertion is proven in \cite[A.4 Theorem 2]{Brown1978}.
\end{proof}
\begin{Remark}
Here $\gplus{\O(d_1)}{\SG}$ is understood to be $\O(d)$ if $d_1=d$ and to be $\SG$ if $d_1=0$.
\end{Remark}

%-------------------------------------------------------------------
%-------------------------------------------------------------------
\section{Translational sections}\label{section:sections}

Our first aim will be to efficiently describe the discrete group $\G<\E(d)$ in terms of the range $\SG$ and the kernel $\F$ of $\pi|_\G$. An important step will be to fix and analyze a section $\T \subset \G$ of the translation group $\T_\SG$ of $\SG$. The main result in this section is Theorem~\ref{Theorem:Ring} which characterizes $m\in\N$ for which $\T^m \triangleleft \G$. 
\begin{Definition}\label{Definition:discreteisometrygroup}
Let $d\in\N$ be the dimension.
Let $d_1,d_2\in\N_0$ be such that $d=d_1+d_2$.
Let $\SG$ be a $d_2$-dimensional space group.
Let $\G<{\gplus{\O(d_1)}{\SG}}$ be discrete such that $\pi(\G)=\SG$, where $\pi$ is the natural surjective homomorphism $\gplus{\O(d_1)}{\E(d_2)}\to\E(d_2)$, $\gplus Ag\mapsto g$.
Let $\F$ be the kernel of $\pi|_\G$ and $\T_\SG$ be the subgroup of translations of $\SG$.
Let $\T\subset\G$ such that the map $\T\to\T_\SG$, $g\mapsto \pi(g)$ is bijective.
\end{Definition}
\begin{Remark}%\label{Remark:Definition}
\begin{enumerate}
\item By Theorem~\ref{Theorem:Browndecomposablediscretegroup} for every discrete group $\G'<\E(d)$ there exists some discrete group $\G$ as in Definition~\ref{Definition:discreteisometrygroup} such that $\G$ is conjugate to $\G'$ under $\E(d)$.
\item If $d_1=0$, we have $d_2=d$, $\G=\SG$, $\T=\T_\SG$ and $\F=\{\id\}$.
If $d_1=d$, we have $d_2=0$, $\G$ is finite, $\G=\F$ and $\T=\{\id\}$.
\item The quantities $d$, $d_1$, $d_2$, $\F$, $\SG$ and $\T_\SG$ are uniquely defined by $\G$.
In general for given $\G$ there is no canonical choice for $\T$, see Example~\ref{Example:choiceofD}.
\item Let $\G$ be given.
In general, for every choice of $\T$ the set $\T$ is not a subset of $\Trans(d)$, see Example~\ref{Example:HelicalGroup}.
Moreover, in general for every choice of $\T$ the set $\T$ is not a group and the elements of $\T$ do not commute, see Example~\ref{Example:choiceofD4}.
\item Let $\G$ be given.
One possible choice for $\T$ is the following.
Let $t_1,\dots,t_{d_2}\in\T_\SG$ be such that $\{t_1,\dots,t_{d_2}\}$ generates $\T_\SG$.
For all $i\in\{1,\dots,d_2\}$ let $g_i\in\G$ such that $\pi(g_i)=t_i$.
Upon this, we define
\[\T=\set{g_1^{n_1}\dots g_{d_2}^{n_{d_2}}}{n_1,\dots,n_{d_2}\in\Z}.\]
\end{enumerate}
\end{Remark}
For the following examples for all angles $\alpha\in\R$ we define the rotation matrix
%
%\label{eq:RotationMatrix}
\[ R(\alpha):=\parens[\bigg]{\begin{matrix}\cos(\alpha)&-\sin(\alpha)\\ \sin(\alpha)&\cos(\alpha)\end{matrix}}\in\O(2).
\]
\begin{Example}[Helical groups]\label{Example:HelicalGroup}
Let $d_1=2$, $d_2=1$, $\alpha\in\R$ be an angle, $n\in\N$,
\[\T=\angles[\Big]{\gplus{R(\alpha)}{\iso{I_1}1}},\quad
\F=\angles[\Big]{\gplus{R(2\pi/n)}{\iso{I_1}0}}\quad\text{and}\quad\P=\angles[\Big]{\parens[\big]{\begin{smallmatrix}1&0\\0&-1\end{smallmatrix}}\oplus\iso{-I_1}0}.\]
Then $\T$ is isomorphic to $\Z$, $\F$ is a cyclic group of order $n$, $\P$ is a group of order 2 and $\F\P$ a dihedral group of order $2n$.
Moreover, $\T$, $\T\F$, $\T\P$ and $\T\F\P$ are decomposable discrete subgroups of $\E(3)$.
If we have $\alpha\in\R\setminus(2\pi\Q)$, the groups $\T$, $\T\F$, $\T\P$ and $\T\F\P$ are so-called \emph{helical groups}, \ie infinite discrete subgroups of the Euclidean group $\E(3)$ which do not contain any translation except the identity.
%, see, \eg, \cite[Chapter 7.4, Theorem 11.13]{Opechowski1986}.
%
\end{Example}
\begin{Example}[The choice of $\T$ is not unique.]\label{Example:choiceofD}
Let $t=\iso{I_1}1$, $\F_0=\{I_2,R(\pi)\}$, $\SG=\T_\SG=\angles t$ and
\[\G=\set[\Big]{\gplus{(R(n\pi/2)F)}{t^n}}{ n\in\Z, F\in \F_0}<\E(3).\]
Then the choice $\gplus{R(\pi/2)}t\in\T$ as well as $\gplus{R(3\pi/2)}t\in\T$ is possible.
In particular, the choice of $\T$ is not unique.
\end{Example}
\begin{Example}\label{Example:choiceofD4}
We present a discrete group $\G<\E(8)$ such that for every choice of $\T$ the set $\T$ is not a group  and the elements of $\T$ do not commute.

Let $\alpha_1,\alpha_2\in\R\setminus(2\pi\Q)$ be angles,
$R_1=R(\alpha_1)$,
$R_2=R(\alpha_2)$,
$R_3=R(\pi/2)$,
$S=(\begin{smallmatrix}1&0\\0 &-1\end{smallmatrix})$,
$t_1=\iso{I_2}{e_1}$ and
$t_2=\iso{I_2}{e_2}$.
Then we have $\angles{R_1}\cong\Z$, $\angles{R_2}\cong\Z$, and $\langle R_3,S\rangle<\O(2)$ is a dihedral group.
Let $\SG=\T_\SG=\set{t_1^{n_1}t_2^{n_2}}{n_1,n_2\in\Z}$,
\[\G:=\set[\bigg]{\gplus{\parens[\Big]{R_1^{n_1}\oplus R_2^{n_2}\oplus(S^{n_1}R_3^{n_2+m})}}{\parens[\big]{t_1^{n_1}t_2^{n_2}}}}{ n_1,n_2\in\Z,m\in\{0,2\}}<\E(8)\]
and $\pi\colon\G\to\SG$ be the natural surjective homomorphism with kernel $\F=\{\id,\gplus{(I_4\oplus R_3^2)}{\id_{\E(2)}}\}$.
Let $\T\subset\G$ such that the map $\T\to\T_\SG$, $g\mapsto\pi(g)$ is bijective.
Since $t_1,t_2\in\T_\SG$, there exist $m_1,m_2\in\{0,2\}$ such that $t_1':=\gplus{(R_1\oplus I_2\oplus (SR_3^{m_1}))}{t_1}\in\T$ and $t_2':=\gplus{(I_2\oplus R_2\oplus R_3^{1+m_2})}{t_2}\in\T$.
We have $t_1't_2'\neq t_2't_1'$ since
\begin{equation}\label{eq:uxw}
t_1't_2'(t_1')^{-1}(t_2')^{-1}=\gplus{\parens[\big]{I_4\oplus(SR_3^{m_1}R_3^{1+m_2}R_3^{-m_1}SR_3^{-1-m_2})}}{\id_{\E(2)}}=\gplus{(I_4\oplus R_3^2)}{\id_{\E(2)}}.
\end{equation}
Thus, the elements of $\T$ do not commute.

Now we suppose that $\T$ is a group.
Since $\pi^{-1}(\id_{\E(2)})=\F$ and by \eqref{eq:uxw}, we have $\pi^{-1}(\id_{\E(2)})\subset\T$.
This contradicts the claim that $\pi|_\T$ is bijective.
Thus, $\T$ is not a group.
\end{Example}
For the remainder of this section we fix the dimension $d\in\N$, the discrete group $\G<\E(d)$ and  the quantities $d_1$, $d_2$, $\T$, $\F$, $\SG$, $\T_\SG$ as introduced by Definition~\ref{Definition:discreteisometrygroup}.
The following lemma collects some elementary properties.
\begin{Lemma}
\begin{enumerate}
\item The group $\F$ is finite.
\item\label{item:abc} For all $n\in\N$ the set $\T^n\F$ is independent of the choice of $\T$, and it holds
\[\T^n\F\triangleleft\G.\]
In particular, it holds $\T\F\triangleleft\G$.
\item The map $\G/\T\F\to\SG/\T_\SG$, $g\T\F\mapsto\pi(g)\T_\SG$ is a group isomorphism, where $\pi\colon\G\to\SG$ is the natural surjective homomorphism with kernel $\F$.
In particular, $\G/\T\F$ is finite.
\item\label{item:abcd} For all $n\in\N$ the map $\T_\SG\to\T^n\F/\F$, $t\mapsto\phi(t^n)\F$ is a group isomorphism, where $\phi\colon\T_\SG^n\to\T^n$ is the canonical bijection.
In particular, the group $\T\F/\F$ is commutative.
\item For all $n\in\Z\setminus\{0\}$ the map $\T\to\T^n$, $t\mapsto t^n$ is bijective.
\end{enumerate}
\label{Lemma:TFP}\end{Lemma}
\begin{proof}
Let $\pi\colon\G\to\SG$ be the natural surjective homomorphism with kernel $\F$.
\begin{enumerate}
\item Since $\G$ is discrete, the group $\F$ is discrete.
Moreover, $\F$ is a subgroup of $\gplus{\O(d_1)}{\{\id_{\E(d_2)}\}}$.
Thus, the group $\F$ is finite.
\item Let $n\in\N$.
The set $\T^n\F$ is the preimage of $\T_\SG^n$ under $\pi$.
Since $\T_\SG^n$ is a normal subgroup of $\SG$, the set $\T^n\F$ is a normal subgroup of $\G$.
\item This is clear, since $\T\F$ is the preimage of $\T_\SG$ under $\pi$.
\item Let $n\in\N$.
Since $\T_\SG$ is isomorphic to $\Z^{d_2}$, the map $\phi_1\colon\T_\SG\to\T_\SG^n$, $t\mapsto t^n$ is a group isomorphism.
Since $\F$ is the kernel of $\pi$ and $\T^n\F$ the preimage of $\T_\SG^n$ under $\pi$, the map $\phi_2\colon\T^n\F/\F\to\T_\SG^n$, $g\F\mapsto\pi(g)$ is an isomorphism.
This implies the assertion, \ie the map $\phi_2^{-1}\circ\phi_1$ is an isomorphism.
\item Let $n\in\Z\setminus\{0\}$.
The map $\psi\colon\T\to\T^n$, $t\mapsto t^n$ is surjective.
Since the map $\T_\SG\to\T_\SG^n$, $t\mapsto t^n$ is injective, the map $\psi$ is injective and thus, bijective.\qedhere
\end{enumerate}
\end{proof}
If $\T$ is a group, it is naturally isomorphic to $\T_\SG$.
\begin{Lemma}\label{Lemma:helptranslationgroup}
Let $m\in\Z\setminus\{0\}$ such that $\T^m$ is a group.
Then, the map
\begin{align*}
\T_\SG\to\T^m, \quad 
t\mapsto\phi(t)^m
\end{align*}
is a group isomorphism, where $\phi\colon\T_\SG\to\T$ is the canonical bijection.
In particular, $\T^m$ is isomorphic to $\Z^{d_2}$.

Furthermore, for all $n\in\Z$ it holds
\[\T^{nm}\triangleleft\T^m.\]
\end{Lemma}
\begin{proof}
Let $m\in\Z\setminus\{0\}$ such that $\T^m$ is a group.
Let $\pi\colon\T\F\to\T_\SG$ be the natural surjective homomorphism with kernel $\F$.
Let $\phi$ be the inverse function of $\pi|_\T$, \ie $\phi\colon\T_\SG\to\T$ is the canonical bijection.
The map
\[\psi_1\colon\T_\SG\to\T\F/\F,\quad t\mapsto\phi(t)\F\]
is an isomorphism.
Since $\T\F/\F$ is isomorphic to $\Z^{d_2}$ and $(\T\F/\F)^m=\T^m\F/\F$, the map
\[\psi_2\colon\T\F/\F\to\T^m\F/\F,\quad t\mapsto t^m\]
is an isomorphism.
Since $\T^m$ is a group, the map
\[\psi_3\colon\T^m\to\T^m\F/\F,\quad g\mapsto g\F\]
is an isomorphism.
The map
\[\T_\SG\to\T^m,\quad t\mapsto\phi(t)^m\]
is equal to $\psi_3^{-1}\circ\psi_2\circ\psi_1$ and thus, an isomorphism.

Let $n\in\Z$.
Since $\T^m$ is isomorphic to $\Z^{d_2}$, we have $\T^{mn}=(\T^m)^n\triangleleft\T^m$.
\end{proof}
We proceed to show that, albeit $\T$ is not a group in general, the situation is much better for special powers of $\T$.
\begin{Definition}\label{Definition:SubsetN}
We define the set
\[\MM:=\set{m\in\N }{ \T^m \text{ is a normal subgroup of }\G}.\]
\end{Definition}
Thus, the quotient group $\G/\T^m$ is well-defined if and only if $m\in \MM$.
\begin{Proposition}\label{Proposition:center2}
For all $m\in \MM$ the group $\T^m$ is a subgroup of the center of $\T\F$.
\end{Proposition}
\begin{proof}
Let $m\in \MM$, $t\in\T$ and $g\in\T\F$.
By Lemma~\ref{Lemma:TFP}\ref{item:abcd} there exists some $f\in\F$ such that
\[gt^m=t^mgf.\]
Since $m\in \MM$, it follows
\[f=g^{-1}t^{-m}gt^m\in\T^m.\]
Since $\T^m\cap\F=\{\id\}$, we have $f=\id$, \ie $g$ and $t^m$ commute.
\end{proof}
\begin{Lemma}\label{Lemma:MMnontrivial}
The set $\MM$ is not empty.
\end{Lemma}
\begin{proof}
Since $\F$ is a normal subgroup of $\G$, for all $g\in\G$ the map
\[\phi_g\colon\F\to\F,\quad f\mapsto g^{-1}fg\]
is a group automorphism.
Let $n$ be the order of the automorphism group of $\F$.
For all $g\in\G$ it holds $\phi_g^{n}=\id$.
Thus for all $g\in\G$ and $f\in\F$ we have
\begin{equation}\label{eq:xa}
g^{n}f=fg^{n},
\end{equation}
\ie $g^{n}$ and $f$ commute.

Now we show that for all $g,h\in\T\F$ the elements $g^{n\abs\F}$ and $h$ commute.
Let $g,h\in\T\F$.
Since $\T\F/\F$ is commutative, there exists some $f\in\F$ such that
\[h^{-1}g^{n}h=g^{n}f.\]
With \eqref{eq:xa} it follows
\begin{equation}\label{eq:xb}
h^{-1}g^{n\abs\F}h=(h^{-1}g^{n}h)^{\abs\F}=(g^{n}f)^{\abs\F}=g^{n\abs\F}f^{\abs\F}=g^{n\abs\F}.
\end{equation}
Now we show that $\T^{n\abs\F^2}$ is a subgroup of $\T\F$.
Let $t,s\in\T$.
We have to show that $t^{n\abs\F^2}s^{-n\abs\F^2}\in\T^{n\abs\F^2}$.
Let $r\in\T$ and $f\in\F$ such that $ts^{-1}=rf$.
Since $\T\F/\F$ is commutative, there exists some $e\in\F$ such that $t^{n\abs\F}s^{-n\abs\F}=r^{n\abs\F}e$.
By \eqref{eq:xb} and \eqref{eq:xa} we have
\[t^{n\abs\F^2}s^{-n\abs\F^2}=(t^{n\abs\F}s^{-n\abs\F})^{\abs\F}=(r^{n\abs\F}e)^{\abs\F}=r^{n\abs\F^2}e^{\abs\F}=r^{n\abs\F^2}\in\T^{n\abs\F^2}.\]
Now we show that $\T^{n\abs\F^2}$ is a normal subgroup of $\G$.
Let $g\in\G$ and $t\in\T$.
We have to show that
\[g^{-1}t^{n\abs\F^2}g\in\T^{n\abs\F^2}.\]
Since $\T^{n}\F$ is a normal subgroup of $\G$, there exist some $s\in\T$ and $f\in\F$ such that
\[g^{-1}t^{n}g=s^{n}f.\]
By \eqref{eq:xa} we have
\[g^{-1}t^{n\abs\F^2}g=(g^{-1}t^{n}g)^{\abs\F^2}=(s^{n}f)^{\abs\F^2}
=s^{n\abs\F^2}f^{\abs\F^2}=s^{n\abs\F^2}\in\T^{n\abs\F^2}.\qedhere\]
\end{proof}
The following Theorem is a key observation on the structural decomposition of $\G$. Not only do we have the existence of exponents $m \in \N$ such that $\T^m \triangleleft \G$ as guaranteed by Lemma~\ref{Lemma:MMnontrivial}, but the set of such `good' exponents will in fact be of the form $m_0\N$, $m_0\in\N$.
\begin{Theorem}\label{Theorem:Ring}
There exists a unique $m_0\in\N$ such that $\MM=m_0\N$.
\end{Theorem}
\begin{proof}
We define the set
\[\widetilde \MM:=\set{m\in\Z}{\T^m\text{ is a normal subgroup of }\G}.\]
First we show that $\widetilde \MM$ is a subgroup of the additive group of integers $\Z$.
It is clear that $0\in\widetilde \MM$.
Let $n_1,n_2\in\widetilde \MM$.
We have to show that $n_1-n_2\in\widetilde \MM$.
Let $\phi\colon\T_\SG\to\T$ be the canonical bijection.
By Proposition~\ref{Proposition:center2} and Lemma~\ref{Lemma:helptranslationgroup}, for all $t,s\in\T_\SG$ it holds
\begin{align*}
\phi(t)^{n_1-n_2}\phi(s)^{-(n_1-n_2)}&=\phi(t)^{n_1}\phi(s)^{-n_1}\phi(t)^{-n_2}\phi(s)^{n_2}=\phi(ts^{-1})^{n_1}\phi(ts^{-1})^{-n_2}\\
&=\phi(ts^{-1})^{n_1-n_2}\in\T^{n_1-n_2},
\end{align*}
and thus, $\T^{n_1-n_2}$ is a group.
It remains to show that $\T^{n_1-n_2}$ is a normal subgroup of $\G$.
Without loss of generality we assume that $n_1,n_2\neq0$, \ie $n_1n_2\neq0$.
Let $g\in\G$ and $t\in\T$.
Since $\T^{n_1},\T^{n_2}\triangleleft\G$, there exist some $s_1,s_2\in\T$ such that $gt^{n_1}g^{-1}=s_1^{n_1}$ and $gt^{n_2}g^{-1}=s_2^{n_2}$.
Since $s_1^{n_1n_2}=gt^{n_1n_2}g^{-1}=s_2^{n_1n_2}$ and the map $\T\to\T^{n_1n_2}$, $r\mapsto r^{n_1n_2}$ is bijective, it holds $s_1=s_2$.
Now we have
\[gt^{n_1-n_2}g^{-1}=(gt^{n_1}g^{-1})(gt^{n_2}g^{-1})^{-1}=s_1^{n_1-n_2}\in\T^{n_1-n_2}.\]
By Lemma~\ref{Lemma:MMnontrivial} and since $\MM\subset\widetilde \MM$, the group $\widetilde \MM$ is nontrivial.
Since every nontrivial subgroup of $\Z$ is equal to $n\Z$ for some $n\in\N$, there exists a unique $m_0\in\N$ such that $\widetilde \MM=m_0\Z$. 
Now, we have
\[\MM=\widetilde \MM\cap\N=m_0\N.\qedhere\]
\end{proof}
\begin{Remark}
\begin{enumerate}
\item The proof of Lemma~\ref{Lemma:MMnontrivial} shows that $m_0$ divides $\abs\F^2\abs{\Aut(\F)}$, where $m_0\in\N$ is such that $\MM=m_0\N$ and $\Aut(\F)$ is the automorphism group of $\F$. In particular, we have an upper bound for $m_0$.
\item\label{item:tk} The group $\G$ is virtually abelian since for all $m\in \MM$ the index of the abelian subgroup $\T^m$ in $\G$ is $m^{d_2}\abs{\F}\abs{\G/(\T\F)}$ and thus, finite. 
\end{enumerate}
\label{Remark:tk}\end{Remark}
%

%-------------------------------------------------------------------
%-------------------------------------------------------------------

%-------------------------------------------------------------------
%-------------------------------------------------------------------
\section{Wave vector characterization of the dual space}\label{subsection:dual}
We now study representations of $\T\F$ and their induced representations on $\G$. As $\T\F = \pi^{-1}(\T_\mathcal{S})$, where $\pi : \G \to \mathcal{S}$ is the natural surjective homomorphism, we are led to first considering characters on $\T\F$ by lifting those on $\T_\mathcal{S}$ via $\pi^{-1}$. (The reader is referred to Section \ref{subsection:harmonic} in the appendix for basic material on representations of a discrete subgroup $\H$ of $\E(d)$, its dual space $\dual \H$ and representations induced by a normal subgroup of finite index.)
\begin{Definition}\label{Definition:ChiK}
For all $k\in\R^{d_2}$ we define the one-dimensional representation $\chi_k\in\dual{\T\F}$ by
\[\chi_k(g):=\exp(2\pi\iu\scalar k{\trans(\pi(g))})\qquad \text{for all }g\in\T\F,\]
where $\pi\colon\T\F\to\T_\SG$ is the natural surjective homomorphism.
\end{Definition}
Since $\T\F$ is a normal subgroup of $\G$, $\G$ acts on $\dual{\T\F}$, {cf.\ Definition~\ref{Definition:Representation}}.
\begin{Lemma}\label{Lemma:characters}
For all $g\in\G$ and $k,k'\in\R^{d_2}$ it holds
\begin{align*}
\chi_k\chi_{k'}=\chi_{k+k'}
\qquad \text{and} \qquad 
g\cdot\chi_k=\chi_{\rot(\pi(g))k},
\end{align*}
where $\pi\colon\G\to\SG$ is the natural surjective homomorphism.
\end{Lemma}
\begin{proof}
Let $g\in\G$, $k,k'\in\R^{d_2}$ and $\pi\colon\G\to\SG$ be the natural surjective homomorphism.
For all $h\in\T\F$ it holds
\begin{align*}
\chi_k(h)\chi_{k'}(h)&=\exp(2\pi\iu\scalar k{\trans(\pi(h))})\exp(2\pi\iu\scalar{k'}{\trans(\pi(h))})\\
&=\exp(2\pi\iu\scalar{k+k'}{\trans(\pi(h))})\\
&=\chi_{k+k'}(h)
\end{align*}
and
\begin{align*}
(g\cdot\chi_k)(h)&=\chi_k(g^{-1}hg)\\
&=\exp(2\pi\iu\scalar k{\trans(\pi(g^{-1}hg))})\\
&=\exp(2\pi\iu\scalar k{\rot(\pi(g^{-1}))\trans(\pi(h))})\\
&=\exp(2\pi\iu\scalar{\rot(\pi(g))k}{\trans(\pi(h))})\\
&=\chi_{\rot(\pi(g))k}(h).\qedhere
\end{align*}
\end{proof}
In our analysis of periodic representations below it will necessary to analyze representations that are trivial on $\T^n$, $n\in\N$. To this end, we recall that a set $L\subset\R^n$ is a \emph{lattice} if $L$ is a subgroup of the additive group $\R^n$ which is isomorphic to the additive group $\Z^n$, and which spans the real vector space $\R^n$. The \emph{dual lattice} $L^*$ (also called the \emph{reciprocal lattice}) of a lattice $L\subset\R^n$ is the set
\[\set{x\in\R^n}{\scalar xy\in\Z\text{ for all }y\in L}.\]
This is indeed a lattice as well, see, \eg, \cite[Chapter 1]{Martinet2003}.
\begin{Definition}\label{Definition:LatticeSpaceGroup}
We define the lattice
\[\LL:=\trans(\T_\SG)<\R^{d_2}\]
and denote its dual lattice by $\dualL$.
\end{Definition}
\begin{Lemma}\label{Lemma:VNK}
For all $n\in\N$ it holds
\[\dualL/n=\set{k\in\R^{d_2}}{\chi_k|_{\T^n}=1}.\]
\end{Lemma}
\begin{proof}
Let $n\in\N$ and $\pi\colon\T\F\to\T_\SG$ be the natural surjective homomorphism.
First we show that $\dualL/n\subset\set{k\in\R^{d_2}}{\chi_k|_{\T^n}=1}$.
Let $k\in \dualL/n$.
For all $t\in\T$ it holds $\trans(\pi(t^n))=n\trans(\pi(t))$ and thus,
\[\chi_k(t^n)=\exp(2\pi\iu\scalar k{\trans(\pi(t^n))})=\exp(2\pi\iu\scalar{nk}{\trans(\pi(t))})=1.\]
Now we show that $\set{k\in\R^{d_2}}{\chi_k|_{\T^n}=1}\subset\dualL/n$.
Let $k\in\R^{d_2}$ such that $\chi_k|_{\T^n}=1$.
Let $x\in\LL$.
There exists some $t\in\T$ such that $x=\trans(\pi(t))$.
We have
\[\scalar{nk}x=\scalar{nk}{\trans(\pi(t))}=\scalar k{\trans(\pi(t^n))}\in\Z,\]
where we used that $\chi_k(t^n)=1$ in the last step.
Since $x\in\LL$ was arbitrary, we have $k\in\dualL/n$.
\end{proof}
In view of $\G$ and the characters of $\T\F$ acting on $\dual{\T\F}$  we proceed to introduce the following equivalence relation. 
\begin{Definition}\label{Definition:RelationTF}
We define the relation $\simrg$ on $\dual{\T\F}$ by
\[(\rho\simrg\rho')\;:\Longleftrightarrow\;(\exists\,g\in\G\,\exists\,k\in\R^{d_2} : g\cdot\rho=\chi_k\rho').\]
\end{Definition}
\begin{Remark}
One can also define an equivalence relation $\simrg$ on the set of all representations of $\T\F$ by
\[(\rho\sim\rho')\;:\Longleftrightarrow\;([\rho]\sim[\rho'])\qquad\text{for all representations $\rho$, $\rho'$ on $\T\F$.}\]
\end{Remark}
\begin{Lemma}
The relation $\simrg$ on $\dual{\T\F}$ is an equivalence relation.
\end{Lemma}
\begin{proof}
It is clear that $\simrg$ is reflexive.

Now we show that $\simrg$ is symmetric.
Let $\rho,\rho'\in\dual{\T\F}$ such that $\rho\simrg\rho'$.
There exist some $g\in\G$ and $k\in\R^{d_2}$ such that $g\cdot\rho=\chi_k\rho'$.
This implies
\[g^{-1}\cdot\rho'=(g^{-1}\cdot\chi_{-k})(g^{-1}\cdot(\chi_k\rho'))=\chi_{-\rot(\pi(g^{-1}))k}\rho,\]
where $\pi\colon\G\to\SG$ is the natural surjective homomorphism.

Now we show that $\simrg$ is transitive.
Let $\rho,\rho',\rho''\in\dual{\T\F}$ such that $\rho\simrg\rho'$ and $\rho'\simrg\rho''$.
There exist some $g,g'\in\G$ and $k,k'\in\R^{d_2}$ such that $g\cdot\rho=\chi_k\rho'$ and $g'\cdot\rho'=\chi_{k'}\rho''$.
This implies
\[(g'g)\cdot\rho=g'\cdot(\chi_k\rho')=\chi_{\rot(\pi(g'))k+k'}\rho'',\]
where $\pi\colon\G\to\SG$ is the natural surjective homomorphism.
\end{proof}
\begin{Definition}\label{Definition:QuotientGroup}
For all groups $\H\le\G$ and $N\in\MM$ such that $\T^N$ is a normal
subgroup of $\H$, let $\H_N$ denote the quotient group $\H/\T^N$.
\end{Definition}
The following lemma gives an algorithm how we can determine a representation set of $\dual{\T\F}/{\simrg}$.
\begin{Lemma}\label{Lemma:RepSys}
Let $m\in\N$ such that $\MM=m\N$.
\begin{enumerate}
\item\label{item:aaa} Every representation set of $\set{\rho\in\dual{\T\F}}{\rho|_{\T^m}=I_{d_\rho}}/{\simrg}$ is a representation set of $\dual{\T\F}/{\simrg}$.
\item\label{item:bbb} The map
\[\dual{(\T\F)_m}\to\set{\rho\in\dual{\T\F}}{\rho|_{\T^m}=I_{d_\rho}},\quad\rho\mapsto\rho\circ\pi\]
where $\pi\colon\T\F\to(\T\F)_m$ is the natural surjective homomorphism, is bijective.
In particular, the set $\set{\rho\in\dual{\T\F}}{\rho|_{\T^m}=I_{d_\rho}}$ is finite.
\item\label{item:ccc} Let $K$ be a representation set of $(\dualL/m)/\dualL$ and $\P$ be a representation set of $\G/(\T\F)$.
Then, for all $\rho,\rho'\in\set{\tilde\rho\in\dual{\T\F}}{\tilde\rho|_{\T^m}=I_{d_{\tilde\rho}}}$ it holds
\[(\rho\simrg\rho')\iff(\exists\,g\in\P\,\exists\,k\in K: g\cdot\rho=\chi_k\rho').\]
\end{enumerate}
\end{Lemma}
\begin{proof}
Let $m\in\N$ such that $\MM=m\N$.

\ref{item:aaa}
Let $R$ be a representation set of $\set{\rho\in\dual{\T\F}}{\rho|_{\T^m}=I_{d_\rho}}/{\simrg}$.
We have to show that for all $\rho\in\dual{\T\F}$ there exists some $\rho'\in R$ such that $\rho\simrg\rho'$.
Let $\rho\in\dual{\T\F}$. 
By Proposition~\ref{Proposition:center2} the group $\T^m$ is a subgroup of the center of $\T\F$ and thus, by Proposition~\ref{Proposition:commutant} for all $t\in\T^m$ there exists some $\lambda\in\C$ such that $\abs\lambda=1$ and $\rho(t)=\lambda I_{d_\rho}$.
Hence, there exists some one-dimensional representation $\chi\in\dual{\T^m}$ such that $\rho|_{\T^m}=\chi I_{d_\rho}$.

There exists some $k\in\R^{d_2}$ such that $\chi|_{\T^m}=\chi_k|_{\T^m}$: By Lemma~\ref{Lemma:helptranslationgroup} the group $\T^m$ is isomorphic to $\Z^{d_2}$.
Thus, there exist $t_1,\dots,t_{d_2}\in\T^m$ such that $\{t_1,\dots,t_{d_2}\}$ generates $\T^m$.
For all $j\in\{1,\dots,d_2\}$ there exists some $\alpha_j\in\R$ such that $\exp(2\pi\iu\alpha_j)=\chi(t_j)$.
For all $i\in\{1,\dots,d_2\}$ let $b_i\in\R^{d_2}$ such that
\[\scalar{b_i}{\trans(\pi(t_j))}=\delta_{ij}\qquad\text{for all }j\in\{1,\dots,d_2\},\]
where $\pi\colon\T\F\to\T_\SG$ is the natural surjective homomorphism.
For $k=\sum_{i=1}^{d_2}\alpha_ib_i\in\R^{d_2}$ it holds $\chi|_{\T^m}=\chi_k|_{\T^m}$.

Thus, we have $\rho|_{\T^m}=\chi_k|_{\T^m}I_{d_\rho}$.
Since $\chi_{-k}\rho\in\dual{\T\F}$ and $(\chi_{-k}\rho)|_{\T^m}=I_{d_\rho}$, there exists some $\rho'\in R$ such that $\chi_{-k}\rho\simrg\rho'$.
There exist some $g\in\G$ and $l\in\R^{d_2}$ such that $g\cdot\rho'=\chi_l(\chi_{-k}\rho)$.
This implies $\rho\simrg\rho'$.

\ref{item:bbb}
This is clear by Proposition~\ref{Proposition:homeomorphism} and Remark~\ref{Remark:tk}\ref{item:tk}.

\ref{item:ccc}
Let $\rho,\rho'\in\dual{\T\F}$ such that $\rho|_{\T^m}=I_{d_\rho}$, $\rho'|_{\T^m}=I_{d_{\rho'}}$ and $\rho\simrg\rho'$.
There exist some $g\in\G$ and $k\in\R^{d_2}$ such that $g\cdot\rho=\chi_k\rho'$.
Let $h\in\P$ such that $g\T\F=h\T\F$.
It holds $I_{d_\rho}=(g\cdot\rho)|_{\T^m}=(\chi_k\rho')|_{\T^m}=\chi_k|_{\T^m}I_{d_{\rho'}}$.
This implies $\chi_k|_{\T^m}=1$ and thus, $k\in(\dualL/m)$ by Lemma~\ref{Lemma:VNK}.
Let $l\in K$ such that $l\dualL=k\dualL$.
We have
\[h\cdot\rho=g\cdot\rho=\chi_k\rho'=\chi_l\rho',\]
where we used \eqref{eq:ActionTrivial} in the first step and that $\chi_{k-l}=1$ since $k-l\in\dualL$ in the last step.

The other direction of the assertion is trivial.
\end{proof}
\begin{Corollary}\label{Corollary:finite-rep-set}
The set $\dual{\T\F}/{\simrg}$ is finite.
\end{Corollary}
\begin{proof}
This is clear by Lemma~\ref{Lemma:RepSys}.
\end{proof}
\begin{Definition}\label{Definition:Grho}
For all $\rho\in\dual{\T\F}$ we define the set
\[\G_\rho:=\set[\Big]{\iso{\rot(\pi(g))}k}{g\in\G,k\in\R^{d_2}:g\cdot\rho=\chi_k\rho}\subset\E(d_2),\]
where $\pi\colon\G\to\SG$ is the natural surjective homomorphism.
\end{Definition}
\begin{Proposition}\label{Proposition:SpaceGroupLattice}
For all $\rho\in\dual{\T\F}$ the set $\G_\rho$ is a space group and it holds
\[\dualL\le\set[\big]{k\in\R^{d_2}}{\iso{I_{d_2}}k\in\G_\rho}\le\dualL/m,\]
where $m\in\N$ is such that $\MM=m\N$.
\end{Proposition}
\begin{proof}
Let $\rho\in\dual{\T\F}$ and $m\in\N$ such that $\MM=m\N$.
First we show that $\G_\rho$ is a subgroup of $\E(d_2)$.
Let $g_1,g_2\in\G_\rho$.
We have to show that $g_1g_2^{-1}\in\G_\rho$.
Let $\pi\colon\G\to\SG$ be the natural surjective homomorphism.
For all $i\in\{1,2\}$ let $h_i\in\G$ and $k_i\in\R^{d_2}$ such that $g_i=\iso{\rot(\pi(h_i))}{k_i}$ and $h_i\cdot\rho=\chi_{k_i}\rho$.
It holds
\[(h_1h_2^{-1})\cdot\rho=h_1\cdot(h_2^{-1}\cdot\rho)=h_1\cdot((h_2^{-1}\cdot\chi_{-k_2})\rho)=((h_1h_2^{-1})\cdot\chi_{-k_2})(h_1\cdot\rho)=\chi_{k_1-\rot(\pi(h_1h_2^{-1}))k_2}\rho\]
and thus,
\[g_1g_2^{-1}=\iso{\rot(\pi(h_1h_2^{-1}))}{k_1-\rot(\pi(h_1h_2^{-1}))k_2}\in\G_\rho.\]
Let
\[\H:=\G_\rho\cap\Trans(d_2)\]
be the group of all translations of $\G_\rho$.
It is clear that $\trans(\H)=\set{k\in\R^{d_2}}{\iso{I_{d_2}}k\in\G_\rho}$.

Now we show that $\trans(\H)\le\dualL/m$.
Let $k\in\trans(\H)$, \ie $\iso{I_{d_2}}k\in\G_\rho$.
There exists some $g\in\G$ such that $g\cdot\rho=\chi_k\rho$ and $\rot(\pi(g))=I_{d_2}$.
The latter implies $\pi(g)\in\T_\SG$ and thus, $g\in\T\F$.
By \eqref{eq:ActionTrivial} we have $\rho=\chi_k\rho$.
Let $\tilde\rho$ be a representative of $\rho$.
There exists some $T\in\U(d_\rho)$ such that $T^H\tilde\rho(g)T=\chi_k(g)\tilde\rho(g)$ for all $g\in\T\F$.
Moreover, by Proposition~\ref{Proposition:center2} the set $\T^m$ is a subset of the center of $\T\F$ and hence, by Proposition~\ref{Proposition:commutant} $\tilde\rho(g)$ is a scalar multiple of $I_{d_\rho}$ for all $g\in\T^m$.
Hence, we have $\chi_k(g)=1$ for all $g\in\T^m$ and $k\in\dualL/m$ by Lemma~\ref{Lemma:VNK}.

Now we show that $\dualL\le\trans(\H)$.
Let $k\in\dualL$.
By Lemma~\ref{Lemma:VNK} we have $\chi_k|_\T=1$.
Since we also have $\chi_k|_\F=1$, we have $\chi_k=1$.
Thus we have $\id_\G\cdot\rho=\chi_k\rho$ and $\iso{I_{d_2}}k\in\H$, \ie $k\in\trans(\H)$.

Now we show that $\G_\rho$ is discrete.
Since $\trans(\H)$ is a subgroup of $\dualL/m$, the group $\H$ is discrete.
Since $\rot(\G_\rho)$ is a subgroup of the finite group $\rot(\SG)$, the index $\abs{\G_\rho:\H}=\abs{\rot(\G_\rho)}$ is finite. Hence, the group $\G_\rho$ is discrete (see, \eg, \cite[Theorem 7.1]{Opechowski1986}).
Since $\dualL$ is a subgroup of $\trans(\H)$, the group $\G_\rho$ contains $d_2$ linearly independent translations.
By Theorem \ref{Theorem:spacegroup} the group $\G_\rho$ is a space group.
\end{proof}
\begin{Lemma}\label{Lemma:VN}
For all $N\in \MM$ and $\rho\in\dual{\T\F}$ such that $\rho|_{\T^N}=I_{d_\rho}$, the set $\dualL/N$ is invariant under $\G_\rho$, \ie $\set{g\cdot k}{g\in\G_\rho,k\in\dualL/N}=\dualL/N$.
\end{Lemma}
\begin{proof}
Let $N\in \MM$ and $\rho\in\dual{\T\F}$ such that $\rho|_{\T^N}=I_{d_\rho}$.
Let $k\in\dualL/N$ and $g\in\G_\rho$.
We have to show that $g\cdot k\in\dualL/N$.
Let $\pi\colon\G\to\SG$ be the natural surjective homomorphism.
There exist some $h\in\G$ and $l\in\R^{d_2}$ such that $g=\iso{\rot(\pi(h))}l$ and $h\cdot\rho=\chi_l\rho$.
Since $\rho|_{\T^N}=I_{d_\rho}=(h\cdot\rho)|_{\T^N}$, we have $\chi_l|_{\T^N}=1$.
We have
\[\chi_{g\cdot k}=\chi_{\rot(\pi(h))k+l}=(h\cdot\chi_k)\chi_l\]
and thus, $\chi_{g\cdot k}|_{\T^N}=1$.
By Lemma~\ref{Lemma:VNK} we have $g\cdot k\in\dualL/N$.
\end{proof}
\begin{Definition}
Let $\H$ be a subgroup of $\E(n)$.
Then the set of all orbits of $\R^n$ under the action of $\H$ is written as $\R^n/\H$ and is called the \emph{quotient of the action} or \emph{orbit space}.
\end{Definition}
\begin{Remark}
If a group $\H<\E(n)$ is discrete, then the quotient space $\R^n/\H$ equipped with the \emph{orbit space distance function}
\[\R^n/\H\times\R^n/\H\to[0,\infty),\qquad (x,y)\mapsto\dist(x,y)\]
is a metric space whose topology is equal to the \emph{quotient topology}, see, \eg,  \cite[\S6.6]{Ratcliffe2006}.
\end{Remark}
We are finally in a position to state and prove our main results on the structure of $\Ind_{\T\F}^\G(\dual{\T\F})$. Note that the set $R$ in the following theorems is finite due to Corollary~\ref{Corollary:finite-rep-set}.
\begin{Theorem}\label{Theorem:RepSet}
Let $R$ be a representation set of $\dual{\T\F}/{\simrg}$.
Then, the map
\begin{align*}
&\bigsqcup_{\rho\in R}\R^{d_2}/\G_\rho\to\Ind_{\T\F}^\G(\dual{\T\F}),\\
&(\G_\rho\cdot k,\rho)\mapsto\Ind_{\T\F}^\G(\chi_k\rho),
\end{align*}
where $\bigsqcup$ is the disjoint union, is bijective.
\end{Theorem}
\begin{proof}
Let $R$ be a representation set of $\dual{\T\F}/{\simrg}$.
We define the map
\begin{align*}
\phi\colon\bigsqcup_{\rho\in R}\R^{d_2}/\G_\rho\to\Ind(\dual{\T\F}), \qquad
(\G_\rho\cdot k,\rho)\mapsto\Ind(\chi_k\rho).
\end{align*}
First we show that $\phi$ is well-defined.
Let $\rho\in R$, $k,k'\in\R^{d_2}$ and $g\in\G_\rho$ such that $k'=g\cdot k$.
Let $\pi\colon\G\to\SG$ be the natural surjective homomorphism.
There exist some $h\in\G$ and $l\in\R^{d_2}$ such that $g=\iso{\rot(\pi(h))}l$ and $h\cdot\rho=\chi_l\rho$.
We have
\[h\cdot(\chi_k\rho)=(h\cdot\chi_k)(h\cdot\rho)=\chi_{\rot(\pi(h))k+l}\rho=\chi_{k'}\rho\]
and thus, $\Ind(\chi_k\rho)=\Ind(\chi_{k'}\rho)$ by Proposition~\ref{Proposition:Induced}.

Now we show that $\phi$ is injective.
Let $\rho,\rho'\in\RR$ and $k,k'\in\R^{d_2}$ such that $\Ind(\chi_k\rho)=\Ind(\chi_{k'}\rho')$.
We have to show that $\rho=\rho'$ and $\G_\rho\cdot k=\G_{\rho'}\cdot k'$.
By Proposition~\ref{Proposition:Induced} there exists some $g\in\G$ such that $g\cdot(\chi_k\rho)=\chi_{k'}\rho'$.
This is equivalent to $g\cdot\rho=\chi_{k'-\rot(\pi(g))k}\rho'$, which implies $\rho\simrg\rho'$ and thus, $\rho=\rho'$.
This implies that $\iso{\rot(\pi(g))}{k'-\rot(\pi(g))k}\in\G_\rho$ and thus,
\[\G_\rho\cdot k=\G_{\rho'}\cdot\parens[\big]{\iso{\rot(\pi(g))}{k'-\rot(\pi(g))k}\cdot k}=\G_{\rho'}\cdot k'.\]
Now we show that $\phi$ is surjective.
Let $\rho\in\dual{\T\F}$.
Let $\rho'\in R$ such that $\rho\simrg\rho'$.
There exist some $g\in\G$ and $k\in\R^{d_2}$ such that $g\cdot\rho=\chi_k\rho'$.
By Proposition~\ref{Proposition:Induced} we have
\[\phi((\G_{\rho'}\cdot k,\rho'))=\Ind(\chi_k\rho')=\Ind(g\cdot\rho)=\Ind\rho.\qedhere\]
\end{proof}
There is also a version of this result for periodic representations. 
\begin{Theorem}\label{Theorem:RepSys}
Let $R$ be a representation set of $\set{\rho\in\dual{\T\F}}{\rho|_{\T^m}=I_{d_\rho}}/{\simrg}$, where $m\in\N$ is such that $\MM=m\N$.
Then the maps
\begin{enumerate}
\item\label{item:22} $\begin{aligned}[t]
&\bigsqcup_{\rho\in R}\set{k/N}{k\in\dualL,N\in \MM}/\G_\rho\to\Ind(\set{\rho\in\dual{\T\F}}{\exists\,N\in \MM:\rho|_{\T^N}=I_{d_\rho}})\\
&(\G_\rho\cdot (k/N),\rho)\mapsto\Ind(\chi_{k/N}\rho)
\end{aligned}$
\item\label{item:33} $\begin{aligned}[t]
&\bigsqcup_{\rho\in R}(\dualL/N)/\G_\rho\to\Ind(\set{\rho\in\dual{\T\F}}{\rho|_{\T^N}=I_{d_\rho}})\\
&(\G_\rho\cdot k,\rho)\mapsto\Ind(\chi_k\rho),
\end{aligned}$
\end{enumerate}
where $\bigsqcup$ is the disjoint union, $\Ind=\Ind_{\T\F}^\G$ and $N\in \MM$ in \ref{item:33} is arbitrary, are bijective.
\end{Theorem}
\begin{proof}
Let $m\in\N$ such that $\MM=m\N$ and $R$ be a representation set of $\set{\rho\in\dual{\T\F}}{\rho|_{\T^m}=I_{d_\rho}}/{\simrg}$.
By Lemma~\ref{Lemma:RepSys} the set $R$ is a representation set of $\dual{\T\F}/{\simrg}$.

\ref{item:22} We define the map
\begin{align*}
\psi&\colon\bigsqcup_{\rho\in R}\set{k/N}{k\in\dualL,N\in \MM}/\G_\rho\to\Ind(\set{\rho\in\dual{\T\F}}{\exists\,N\in \MM:\rho|_{\T^N}=I_{d_\rho}})\\
&(\G_\rho\cdot(k/N),\rho)\mapsto\Ind(\chi_{k/N}\rho).
\end{align*}
First we show that $\psi$ is well-defined.
Let $\rho\in R$, $k\in\dualL$ and $N\in \MM$.
Since $\T^N\subset\T^m$ and by Lemma~\ref{Lemma:VNK}, we have $(\chi_{k/N}\rho)|_{\T^N}=I_{d_\rho}$.
By Lemma~\ref{Lemma:VN} for all $N\in \MM$ we have $(\dualL/N)/\G_\rho\subset\R^{d_2}/\G_\rho$ and thus, by Theorem~\ref{Theorem:RepSet} the map $\psi$ is well-defined.

Since the map of Theorem~\ref{Theorem:RepSet} is injective, also $\psi$ is injective.

It remains to show that $\psi$ is surjective.
Let $\rho\in\dual{\T\F}$ and $N\in \MM$ such that $\rho|_{\T^N}=I_{d_\rho}$.
There exists some $\rho'\in R$ such that $\rho\simrg\rho'$.
There exist some $g\in\G$ and $k\in\R^{d_2}$ such that $g\cdot\rho=\chi_k\rho'$.
We have $(g\cdot\rho)|_{\T^N}=I_{d_\rho}=\rho'|_{\T^N}$ and thus, $\chi_k|_{\T^N}=1$.
By Lemma~\ref{Lemma:VNK} we have $k\in\dualL/N$ and thus
\[\psi((\G_{\rho'}\cdot k,\rho'))=\Ind(\chi_k\rho')=\Ind(g\cdot\rho)=\Ind\rho,\]
by Proposition~\ref{Proposition:Induced}.
\ref{item:33} The proof is analogous to the proof of \ref{item:22}.
\end{proof}
As a direct consequence of the above theorems we obtain:
\begin{Corollary}\label{Corollary:Subreps-induced}
Let $R$ be as in Theorem~\ref{Theorem:RepSet}. For every $\sigma\in\dual\G$ there exists a $\rho\in R$ and a $k\in\R^{d_2}$ such that $\sigma$ is a subrepresentation of $\Ind_{\T\F}^\G(\chi_k\rho)$. If moreover $R$ is as in Theorem~\ref{Theorem:RepSys} and $\rho|_{\T^N}=I_{d_\rho}$ for an $N \in M_0$, then $k$ can be chosen in $\dualL/N$. 
\end{Corollary}
\begin{proof}
This is a direct consequence of Theorems~\ref{Theorem:RepSet} and~\ref{Theorem:RepSys} and Proposition~\ref{Proposition:Teildarstellung}.  
\end{proof}

We finally address the natural question to what extend the induced representations labeled by $ \bigsqcup_{\rho\in R}\R^{d_2}/\G_\rho$ from Theorem~\ref{Theorem:RepSet} are irreducible and exhaust the total dual space $\dual\G$.  To this end, we equip the orbit space $\R^{d_2}/\mathcal{H}$ of a $d_2$-dimensional space group $\mathcal{H}$ with the pushforward under the action of $\mathcal{H}$ of the $d_2$-dimensional Lebesgue-measure restricted to an associated fundamental domain. The dual space $\dual\G$ of $\G$ shall be endowed with the Plancherel measure, cf.\ Theorem~\ref{Theorem:Plancherel-measure}. Without loss of generality we restrict to groups of infinite order as for $|\G| < \infty$ one has $d_2=0$, $\T = \{\id\}$ and $\T\F = \G$. 
\begin{Theorem}\label{Theorem:Bijektion}
Suppose $\G$ is of infinite order.
Let $R$ be a representation set of $\set{\rho\in\dual{\T\F}}{\rho|_{\T^m}=I_{d_\rho}}/{\simrg}$, where $m\in\N$ is such that $\MM=m\N$. There are a null-set $N\subset\dual\G$ and null-sets $N_\rho\subset\R^{d_2}/\G_\rho$ for each $\rho\in R$ such that the mapping 
\begin{align*}
&\bigsqcup_{\rho\in R}(\R^{d_2}/\G_\rho)\setminus N_\rho\to\dual\G\setminus N\\
&(\G_\rho\cdot k,\rho)\mapsto\Ind_{\T\F}^\G(\chi_k\rho),
\end{align*}
is bijective. 
\end{Theorem}
\begin{Remark}\label{Remark:Null-Set}
The proof shows that the $N_\rho$ can be chosen as $N_\rho= \G_\rho \cdot K$, where $K$ is the Lebesgue null-set 
\[K = \set[\Big]{k\in\R^{d_2}}{\exists A\in\rot(\SG)\setminus\{I_{d_2}\}:(A-I_{d_2})k\in\dualL/m}.\]
Also note that, by Theorem~\ref{Theorem:RepSet} and Corollary~\ref{Corollary:Subreps-induced}, for any $\sigma\in N$ there is a $\rho\in R$ and a $\G_\rho\cdot k\in N_\rho$ such that $\sigma$ is a subrepresentation of $\Ind_{\T\F}^\G(\chi_k\rho)$. 
\end{Remark}

\begin{Lemma}\label{Lemma:Lattice-Diff-k-k}
Let $K$ be as in Remark~\ref{Remark:Null-Set}. Suppose $g\in\G$ and $k,k'\in\R^{d_2}$ are such that $\rot(\pi(g))k-k' \in\dualL/m$, where $m\in\N$ is such that $\MM=m\N$. 
Then we have $k\in K$ if and only if $k'\in K$.
\end{Lemma}
\begin{proof}
Let $g\in\G$ and $k,k'\in\R^{d_2}$ such that $s\in\dualL/m$, where $B=\rot(\pi(g))$ and $s=Bk-k'$.
Assume $k \in K$, say $(A-I_{d_2})k\in\dualL/m$ with $A\in\rot(\SG)\setminus\{I_{d_2}\}$. Observe that also  $BAB^T \in \rot(\SG)\setminus\{I_{d_2}\}$ and that $\dualL/m$ is invariant under $\rot(\SG)$, since $\rot(\SG)\ltimes\{0_{d_2}\}<\G_{\chi_0}$ and by Lemma~\ref{Lemma:VN} with $\rho = \chi_0$. We thus get 
\[ (BAB^T-I_{d_2})k' 
%   = (BAB^T-I_{d_2})(Bk-s) 
   = B(A-I_{d_2}) k - BAB^T s + s \in\dualL/m, 
\]
\ie, $k' \in K$. If, conversely, $k' \in K$, then the same argument with $g$ replaced by $g^{-1}$ and $s$ replaced by $- B^T s$ yields $k \in K$.  
\end{proof}
As a consequence of this we have the following lemma. 
\begin{Lemma}\label{Lemma:K-Invariance}
The set $K$ defined in Remark~\ref{Remark:Null-Set} is invariant under $\G_\rho$ for any $\rho\in\dual{\T\F}$ with $\rho|_{\T^m}=I_{d_\rho}$, where $m\in\N$ is such that $\MM=m\N$.
\end{Lemma}
\begin{proof}
Let $k \in K$ and $g \in \G_\rho$, where $\rho\in\dual{\T\F}$ with $\rho|_{\T^m}=I_{d_\rho}$. 
There exist some $h\in\G$ and $l\in\R^{d_2}$ such that $g=\iso{\rot(\pi(h))}l$ and $h\cdot\rho=\chi_l\rho$.
Since $\rho|_{\T^m}=I_{d_\rho}=(h\cdot\rho)|_{\T^m}$, we have $\chi_l|_{\T^m}=1$ and so $l \in \dualL/m$ due to Lemma~\ref{Lemma:VNK}. Now Lemma~\ref{Lemma:Lattice-Diff-k-k} implies $g \cdot k = \rot(\pi(h)) k + l \in K$. 
\end{proof}
We now prove Theorem~\ref{Theorem:Bijektion}. 
\begin{proof}
As in the proof of Theorem~\ref{Theorem:RepSys} we let $m\in\N$ such that $\MM=m\N$ and $R$ be a representation set of $\set{\rho\in\dual{\T\F}}{\rho|_{\T^m}=I_{d_\rho}}/{\simrg}$ so that $R$ is a also representation set of $\dual{\T\F}/{\simrg}$ due to Lemma~\ref{Lemma:RepSys}. Choose $K$ and, for each $\rho\in R$, $N_\rho=\G_\rho\cdot K$ as in Remark~\ref{Remark:Null-Set}. We then claim that the following implication holds true for any $\rho\in\dual{\T\F}$ with $\rho|_{\T^m}=I_{d_\rho}$, $k\in\R^{d_2}$ and $g\in\G\setminus\T\F$: 
\[ g\cdot(\chi_k\rho)=\chi_k\rho \implies k\in K. \] 
To see this fix such $\rho$, $k$ and $g$ with $g\cdot(\chi_k\rho)=\chi_k\rho$. Since $\rho|_{\T^m}=I_{d_\rho}$ we also have $g\cdot\chi_k|_{\T^m}=\chi_k|_{\T^m}$. Thus Lemmas~\ref{Lemma:characters} and~\ref{Lemma:VNK} yield $(\rot(\pi(g))-I_{d_2})k\in\dualL/m$ and hence $k\in K$.

Since the point group $\rot(\SG)$ is finite and the lattice $\dualL$ is countable, the Lebesgue measure of $K\subset\R^{d_2}$ is $0$. As also $\G_\rho$ is countable for any $\rho\in R$, the sets $N_\rho\subset\R^{d_2}/\G_\rho$ are null-sets as well.

By Theorem~\ref{Theorem:RepSet} for each $\rho\in R$ and $\G_\rho\cdot k\in\R^{d_2}/\G_\rho$ there is a well-defined representation $\Ind_{\T\F}^\G(\chi_k\rho)$. Combining Mackey's irreducibility criterion stated in Theorem~\ref{Theorem:Mackey} and the above claim we see that $\Ind_{\T\F}^\G(\chi_k\rho)$ is irreducible for every $\rho\in R$ and $\G_\rho\cdot k\in(\R^{d_2}/\G_\rho)\setminus N_\rho$. We thus obtain that the mapping in the assertion of the theorem is well-defined and, upon choosing $N\subset\dual\G$ suitably, also surjective. By Theorem~\ref{Theorem:RepSet} it is injective as well since, as noted above, $R$ is a also representation set of $\dual{\T\F}/{\simrg}$.

We proceed to show that $N$ is a null-set with respect to the Plancherel measure on $\dual\G$. First note that, as $\T^m \cong \Z^{d_2}$, the dual space of $\T^m$ is $\dual{\T^m} = \set{\chi_k|_{\T^m}}{k\in \R^{d_2}}$, which by Lemma~\ref{Lemma:VNK} can be identified with the flat torus $\R^{d_2}/(\dualL/m)$ whose Haar measure is the pushforward under the action of $\dualL/m$ of the $d_2$-dimensional Lebesgue-measure restricted to a unit cell of $\dualL/m$.

Let $\Omega$ be the $\mu_{\T^m}$-conull set $\set{\chi_k|_{\T^m}}{k\in \R^{d_2} \setminus K} \subset \dual{\T^m}$. Note that $\Omega$ is $\G$-invariant: For $g \in \G$ and $k \in \R^{d_2}$ Lemma~\ref{Lemma:characters} gives $g \cdot \chi_k|_{\T^m} = \chi_{k'}|_{\T^m}$ with $k' = \rot(\pi(g)) k$ and then Lemma~\ref{Lemma:Lattice-Diff-k-k} shows that $k \in K$ if and only if $k' \in K$. Trivially, each $\chi_k|_{\T^m}$ extends to $\chi_k$ on $\T\F$ and $g \cdot \chi_k|_{\T^m} = \chi_k|_{\T^m}$ for all $g \in \T\F$. Conversely, if $g \in \G \setminus \T\F$ and $k \in \R^{d_2} \setminus K$, then the above claim yields  $g \cdot \chi_k|_{\T^m} \neq \chi_k|_{\T^m}$. As a consequence, $\G^{\chi_k|_{\T^m}} = \T\F$ and $\G_m^{\chi_k|_{\T^m}} = (\T\F)_m$ whenever $k \in \R^{d_2} \setminus K$. Invoking Theorem~\ref{Theorem:Mackey-machine} we thus find that 
\[\bigcup_{\G \cdot k \notin K/\G}\set[\Big]{\Ind_{\T\F}^\G(\chi_k \rho)}{\rho\in\dual{\T\F} : \rho|_{\T^m}=I_{d_\rho}}, \]
is a $\mu_\G$-conull subset of $\dual\G$, where we also have used Lemma~\ref{Lemma:RepSys}\ref{item:bbb} in order to pass from representations in $\dual{(\T\F)_m}$ to representations in $\dual{\T\F}$ that annihilate $\T^m$. To finish the proof it now suffices to show that 
\[\bigcup_{\G\cdot k\notin K/\G} \set[\Big]{\Ind_{\T\F}^\G(\chi_k\rho)}{\rho\in\dual{\T\F} : \rho|_{\T^m}=I_{d_\rho}}\subset\bigsqcup_{\rho\in R}\set[\big]{\Ind_{\T\F}^\G(\chi_k\rho)}{\G_\rho\cdot k\in(\R^{d_2}/\G_\rho) \setminus N_\rho}.\]
Let $k\in \R^{d_2} \setminus K$ and $\rho\in\dual{\T\F} : \rho|_{\T^m}=I_{d_\rho}$.
Let $\rho'\in R$ such that $\rho'\simrg\chi_k\rho$.
There exists some $g\in\G$ and $k'\in\R^{d_2}$ such that $g\cdot (\chi_k \rho)=\chi_{k'}\rho'$.
We have $\rho|_{\T^m}=I_{d_\rho}=\rho'|_{\T^m}$ and thus we have $\chi_{\rot(\pi(g))k}|_{\T^m}=\chi_{k'}|_{\T^m}$. So with Lemma~\ref{Lemma:VNK} we get $\rot(\pi(g))k-k'\in\dualL/m$ and thus $k'\in \R^{d_2} \setminus K$ due to Lemma~\ref{Lemma:Lattice-Diff-k-k}.  By Lemma~\ref{Lemma:K-Invariance} we then  have $\G_\rho\cdot k' \notin N_\rho$. The claim now follows from 
\[\Ind(\chi_k\rho)=\Ind(g\cdot(\chi_k\rho))=\Ind(\chi_{k'}\rho').\qedhere\]
\end{proof}
%

%-------------------------------------------------------------------
%-------------------------------------------------------------------
\section{Harmonic analysis of periodic functions}\label{subsection:fourier}
In this section we develop methods from Fourier analysis for periodic mappings on $\G$. We begin by defining a suitable notion of periodicity.
\begin{Definition}\label{Definition:periodic}
Let $S$ be a set and $N\in \MM$.
A function $u\colon\G\to S$ is called \emph{$\T^N$-periodic} if
\[u(g)=u(gt)\qquad\text{for all }g\in\G\text{ and }t\in\T^N.\]
A function $u\colon\G\to S$ is called \emph{periodic} if there exists some $N\in \MM$ such that $u$ is $\T^N$-periodic.

We equip $\C^{m\times n}$ with the Frobenius inner product $\scalar\fdot\fdot$ defined by
\[\scalar AB:=\sum_{i=1}^m\sum_{j=1}^n a_{ij}\conj{b_{ij}}\qquad\text{for all }A,B\in\C^{m\times n}\]
and let $\norm\fdot$ denote the induced norm.
We define the set
\[\Per(\G,\C^{m\times n}):=\set{u\colon\G\to\C^{m\times n}}{u\text{ is periodic}}.\]
\end{Definition}
\begin{Remark}
If $\G$ is finite and $S$ a set, then every function from $\G$ to $S$ is periodic and in particular, we have $\Per(\G,\C^{m\times n})=\{u\colon\G\to\C^{m\times n}\}$.
\end{Remark}
The following Lemma shows that the above definition of periodicity is independent of the choice of $\T$.
\begin{Lemma}
Let $S$ be a set.
A function $u\colon\G\to S$ is periodic if and only if there exists some $N\in\N$ such that
\[u(g)=u(gh)\qquad\text{for all }g\in\G\text{ and }h\in\G^N.\]
\end{Lemma}
\begin{proof}
Let $S$ be a set and $u\colon\G\to S$ be $\T^N$-periodic for some $N\in \MM$.
By Theorem~\ref{Theorem:Ring} the function $u$ is $\T^{\abs\F N}$-periodic.
By Proposition~\ref{Proposition:center2} it holds
\[\G^{\abs{\G/(\T\F)}\abs\F N}\subset(\T\F)^{\abs\F N}\subset(\T^N\F)^{\abs\F}=\T^{\abs\F N}\F^{\abs\F}=\T^{\abs\F N}\subset\T^N.\]
and thus, we have
\[u(g)=u(gh)\qquad\text{for all }g\in\G\text{ and }h\in\G^{\abs{\G/(\T\F)}\abs\F N}.\]
The other direction is trivial since by Theorem~\ref{Theorem:Ring} for all $N\in\N$ there exists some $n\in\N$ such that $nN\in \MM$.
\end{proof}
The following lemma characterizes the periodic functions on $\G$ with the aid of the quotient groups $\G/\T^N$.
\begin{Lemma}
If $N\in \MM$ and $u\colon\G\to S$ is $\T^N$-periodic, then the function
\begin{align*}
\G_N\to S, \qquad 
g\T^N\mapsto u(g)
\end{align*}
is well-defined.
Moreover, we have
\[\Per(\G,\C^{m\times n})=\set[\Big]{\G\to\C^{m\times n},g\mapsto u(g\T^N)}{N\in \MM,u\colon\G_N\to\C^{m\times n}},\]
and this space is a vector space. 
\end{Lemma}
\begin{proof}
This follows immediately from the definition of $\Per(\G,\C^{m\times n})$. (Note that, if $u_i\in\Per(\G,\C^{m\times n})$ is $\T^{N_i}$-periodic, $N_i\in \MM$, $i=1,2$, then $u_1+u_2$ is $\T^{N_1N_2}$-periodic.) 
\end{proof}
\begin{Definition}\label{Definition:CCN}
For all $N\in \MM$ let $\CC_N$ be a representation set of $\G/\T^N$.
\end{Definition}
\begin{Remark}
\begin{enumerate}
\item If $\G$ is finite, we have $\CC_N=\G$ for all $N\in \MM$.
\item\label{item:Cube} Let $\G$ be infinite.
There exists some $m\in\N$ such that $\MM=m\N$ and there exist $t_1,\dots,t_{d_2}\in\T^m$ such that  $\{t_1,\dots,t_{d_2}\}$ generates $\T^m$.
Let $\CC$ be a representation set of $\G/\T^m$.
Then for all $N\in \MM$ a feasible choice for $\CC_N$ is
\[\CC_N=\set[\Big]{t_1^{n_1}\dots t_{d_2}^{n_{d_2}}g}{n_1,\dots,n_{d_2}\in\{0,\dots,N/m-1\},g\in\CC}.\]
For this choice, for all $x\in\R^d$ and large $N\in \MM$ the set $\CC_N\cdot x$ is similar to a cube which explains the nomenclature.
\end{enumerate}
\label{Remark:Cube}\end{Remark}
We equip the vector space $\Per(\G,\C^{m\times n})$ with an inner product.
\begin{Definition}\label{Definition:InnerProduct}
We define the inner product $\angles{\fdot,\fdot}$ on $\Per(\G,\C^{m\times n})$ by
\[\angles{u,v}:=\frac1{\abs{\CC_N}}\sum_{g\in\CC_N}\angles{u(g),v(g)}\qquad\text{if $u$ and $v$ are $\T^N$-periodic}\]
for all $u,v\in\Per(\G,\C^{m\times n})$.
We denote the induced norm by $\norm\fdot_2$.
\end{Definition}
In order to define the Fourier transform we must choose a set of representatives for the periodic elements of $\dual\G$.
\begin{Definition}\label{Definition:EE}
Let $\EE$ be a representation set of $\set{\rho\in\dual\G}{\rho\text{ is periodic}}$.
\end{Definition}
\begin{Remark}
\begin{enumerate}
\item All representations of $\EE$ are unitary by Definition~\ref{Definition:Representation} which is necessary for the Plancherel formula in Proposition~\ref{Proposition:TFplancherelmatrix}.
\item For all $N\in \MM$ a representation of $\G$ is $\T^N$-periodic if and only if $\rho|_{\T^N}=I_{d_\rho}$.
\item Proposition~\ref{Proposition:homeomorphism} shows that
\[\set[\big]{\rho\in\dual\G}{\rho\text{ is periodic}}=\set[\big]{\rho\circ\pi_N}{N\in \MM,\rho\in\dual{\G_N}},\]
where $\pi_N$ is the natural surjective homomorphism from $\G$ to $\G_N$ for all $N\in \MM$.
\end{enumerate}
\end{Remark}
\begin{Definition}\label{Definition:FourierPeriodic}
For all $u\in\Per(\G,\C^{m\times n})$ and for all periodic representations $\rho$ of $\G$ we define
\[\fourier u(\rho):=\frac1{\abs{\CC_N}}\sum_{g\in\CC_N}u(g)\otimes\rho(g)\in\C^{(md_\rho)\times(nd_\rho)},\]
where $N\in\MM$ is such that $u$ and $\rho$ are $\T^N$-periodic and $\otimes$ denotes the Kronecker product.
\end{Definition}
\begin{Proposition}[The Plancherel formula]\label{Proposition:TFplancherelmatrix}
The Fourier transformation
\[\fourier\fdot\colon\Per(\G,\C^{m\times n})\to\bigoplus_{\rho\in\EE}\C^{(md_\rho)\times(nd_\rho)},\quad u\mapsto(\fourier u(\rho))_{\rho\in\EE}\]
is well-defined and bijective.
Moreover, we have the Plancherel formula
\[\angles{u,v}=\sum_{\rho\in\EE}d_\rho\angles{\fourier u(\rho),\fourier v(\rho)}\qquad\text{for all }u,v\in\Per(\G,\C^{m\times n}).\]
\end{Proposition}
\begin{proof}
We show that the well-known Plancherel formula for finite groups, see, \eg, \cite[Exercise~5.7]{Steinberg2012}, implies the Plancherel formula of the proposition.
Let $N\in \MM$ and $\pi_N\colon\G\to\G_N$ be the natural surjective homomorphism.
The map
\begin{align*}
f_1\colon\{u\colon\G_N\to\C^{m\times n}\}\to\set{u\in\Per(\G,\C^{m\times n})}{u\text{ is $\T^N$-periodic}}, \qquad 
u\mapsto u\circ\pi_N
\end{align*}
is bijective.
Let $\EE_N=\set{\rho}{\rho\text{ is a representation of }\G_N, \rho\circ\pi_N\in\EE}$.
We have $\set{\rho\circ\pi_N}{\rho\in\EE_N}=\set{\rho\in\EE}{\rho\text{ is $\T^N$-periodic}}$.
Thus the map
\begin{align*}
f_2\colon&\bigoplus_{\rho\in\EE,\ \rho\text{ is $\T^N$-periodic}}\C^{(md_\rho)\times(nd_\rho)}\to\bigoplus_{\rho\in\EE_N}\C^{(md_\rho)\times(nd_\rho)},\\
&\parens{A_\rho}_{\rho\in\EE,\ \rho\text{ is $\T^N$-periodic}}\mapsto\parens{A_{\rho\circ\pi_N}}_{\rho\in\EE_N}
\end{align*}
is bijective.
By Proposition~\ref{Proposition:homeomorphism} the set $\EE_N$ is a representation set of $\dual{\G_N}$.
For all $u\colon\G_N\to\C^{m\times n}$ and $\rho\in\EE_N$ we define $\fourier u(\rho)=\frac1{\abs{\G_N}}\sum_{g\in\G_N}u(g)\otimes\rho(g)$.
By the Plancherel formula for finite groups, the Fourier transformation
\[\fourier\fdot\colon\{u\colon\G_N\to\C^{m\times n}\}\to\bigoplus_{\rho\in\EE_N}\C^{(md_\rho)\times(nd_\rho)}, \qquad u\mapsto (\fourier u(\rho))_{\rho\in\EE_N}\]
is bijective and it holds $\frac1{\abs{\G_N}}\sum_{g\in\G_N}\angles{u(g),v(g)}=\sum_{\rho\in\EE_N}d_\rho\angles{\fourier u(\rho),\fourier v(\rho)}$ for all $u,v\colon\G_N\to\C^{m\times n}$.
The diagram
\[ \begin{tikzcd}
\set{u\in\Per(\G,\C^{m\times n})}{u\text{ is $\T^N$-periodic}} \arrow{r}{\fourier{\,\fdot\,}}  & \bigoplus_{\rho\in\EE,\ \rho\text{ is $\T^N$-periodic}}\C^{(md_\rho)\times(nd_\rho)}  \arrow{d}{f_2} \\
\{u\colon\G_N\to\C^{m\times n}\} \arrow[swap]{u}{f_1} \arrow{r}{\fourier{\,\fdot\,}}& \bigoplus_{\rho\in\EE_N}\C^{(md_\rho)\times(nd_\rho)}
\end{tikzcd}
\]
commutes, where the upper map is defined by $u\mapsto(\fourier u(\rho))_{\rho\in\EE,\ \rho\text{ is $\T^N$-periodic}}$.
Thus, the map
\begin{equation}\label{eq:rsl}
\fourier\fdot\colon\set{u\in\Per(\G,\C^{m\times n})}{u\text{ is $\T^N$-periodic}}\to\bigoplus_{\rho\in\EE,\ \rho\text{ is $\T^N$-periodic}}\C^{(md_\rho)\times(nd_\rho)}
\end{equation}
is bijective and we have
\[\angles{u,v}=\sum_{\rho\in\EE,\ \rho\text{ is $\T^N$-periodic}}d_\rho\angles{\fourier u(\rho),\fourier v(\rho)}\]
for all $\T^N$-periodic functions $u,v\in\Per(\G,\C^{m\times n})$.

Since $N\in \MM$ was arbitrary, for all $u\in\Per(\G,\C^{m\times n})$, for all $N\in \MM$ such that $u$ is $\T^N$-periodic and $n\in\N$ it holds
\begin{equation}\label{eq:onex}
\sum_{\rho\in\EE,\ \rho\text{ is $\T^N$-periodic}}d_\rho\norm{\fourier u(\rho)}^2=\norm u_2^2=\sum_{\rho\in\EE,\ \rho\text{ is $\T^{nN}$-periodic}}d_\rho\norm{\fourier u(\rho)}^2.
\end{equation}
By \eqref{eq:onex} for all $u\in\Per(\G,\C^{m\times n})$ and $N\in \MM$ such that $u$ is $\T^N$-periodic, we have
\begin{equation}\label{eq:oney}
\set{\rho\in\EE}{\fourier u(\rho)\neq 0}\subset\set{\rho\in\EE}{\rho\text{ is $\T^N$-periodic}}.
\end{equation}

By \eqref{eq:onex} and \eqref{eq:oney} the Fourier transformation $\fourier\fdot\colon\Per(\G,\C^{m\times n})\to\bigoplus_{\rho\in\EE}\C^{(md_\rho)\times(nd_\rho)}$ is well-defined and we have
\[\angles{u,v}=\sum_{\rho\in\EE}d_\rho\angles{\fourier u(\rho),\fourier v(\rho)}\]
for all $u,v\in\Per(\G,\C^{m\times n})$.
Moreover, since the map defined in \eqref{eq:rsl} is injective and $\Per(\G,\C^{m\times n})=\bigcup_{N\in\MM}\set{u\in\Per(\G,\C^{m\times n})}{u\text{ is $\T^N$-periodic}}$, the Fourier transformation is injective.
Analogously, the Fourier transformation is surjective.
\end{proof}
\begin{Remark}
\begin{enumerate}
\item The above proof also shows that for all $u\colon\G\to\C^{m\times n}$ and $N\in \MM$ such that $u$ is $\T^N$-periodic, we have
\[\set{\rho\in\EE}{\fourier u(\rho)\neq 0}\subset\set{\rho\in\EE}{\rho\text{ is $\T^N$-periodic}}.\]
Moreover, for all $N\in \MM$ the map
\begin{align*}
&\set[\big]{u\colon\G\to\C^{m\times n}}{u\text{ is $\T^N$-periodic}}\to\bigoplus_{\rho\in\EE,\ \rho\text{ is $\T^N$-periodic}}\C^{(md_\rho)\times (nd_\rho)},\\
&u\mapsto \parens[\big]{\fourier u(\rho)}
\end{align*}
is bijective.
\item It is easy to see that by the above proposition we have also a description of the completion of $\Per(\G,\C^{m\times n})$ with respect to the norm $\norm\fdot_2$.
We have
\[\overline{\Per(\G,\C^{m\times n})}^{\norm\fdot_2}=\set[\bigg]{u\colon\G\to\C^{m\times n}}{\sum_{\rho\in\EE}d_\rho\norm{\fourier u(\rho)}^2<\infty}\]
and the map

\begin{align*}
&\overline{\Per(\G,\C^{m\times n})}^{\norm\fdot_2}\to\set[\bigg]{a\in\prod_{\rho\in\EE}\C^{(md_\rho)\times(nd_\rho)}}{\sum_{\rho\in\EE}d_\rho \norm{a(\rho)}^2<\infty},\\
&u\mapsto(\fourier u(\rho))_{\rho\in\EE}
\end{align*}
is bijective.
\end{enumerate}
\end{Remark}
\begin{Lemma}\label{Lemma:PeriodicTranslation}
Let $f\in\Per(\G,\C^{m\times n})$, $g\in\G$ and $\tau_gf$ denote the translated function $f(\fdot g)$.
Then we have $\tau_gf\in\Per(\G,\C^{m\times n})$ and
\[\fourier{\tau_g f}(\rho)=\fourier f(\rho)(I_n\otimes\rho(g^{-1}))\]
for all periodic representations $\rho$ of $\G$.
\end{Lemma}
\begin{proof}
Let $f\in\Per(\G,\C^{m\times n})$, $g\in\G$ and $\rho$ be a periodic representation.
Let $N\in \MM$ such that $f$ and $\rho$ are $\T^N$-periodic.
The function $\tau_gf$ is $\T^N$-periodic and we have 
\begin{align*}
\fourier{\tau_gf}(\rho)&=\frac1{\abs{\CC_N}}\sum_{h\in\CC_N}\tau_gf(h)\otimes \rho(h)\\
&=\frac1{\abs{\CC_N}}\sum_{h\in\CC_N}f(hg)\otimes\rho(h)\\
&=\frac1{\abs{\CC_N}}\sum_{h\in\CC_N}f(h)\otimes \rho(hg^{-1})\\
&=\frac1{\abs{\CC_N}}\sum_{h\in\CC_N}f(h)\otimes(\rho(h)\rho(g^{-1}))\\
&=\parens[\Big]{\frac1{\abs{\CC_N}}\sum_{h\in\CC_N}f(h)\otimes \rho(h)}(I_n\otimes\rho(g^{-1}))\\
&=\fourier f(\rho)(I_n\otimes\rho(g^{-1})),
\end{align*}
where in the third step we made a substitution and used that $\CC_N$ and $\CC_Ng$ are representation sets of $\G/\T^N$ and that the function $h\mapsto f(h)\otimes\rho(hg^{-1})$ is $\T^N$-periodic.
\end{proof}
\begin{Definition}\label{Definition:FourierLOne}
For all $u\in L^1(\G,\C^{m\times n})$ and all representations $\rho$ of $\G$ we define
\[\fourier u(\rho):=\sum_{g\in\G}u(g)\otimes \rho(g).\]
\end{Definition}
\begin{Remark}
If the group $\G$ is finite, $\rho$ is a representation of $\G$ and $u\in L^1(\G,\C^{m\times n})=\Per(\G,\C^{m\times n})$, then the Definitions~\ref{Definition:FourierPeriodic} and \ref{Definition:FourierLOne} for $\fourier u(\rho)$ differ by the multiplicative constant $\abs\G$, but it will always be clear from the context which of the both definitions is meant.
If $\G$ is infinite, then $L^1(\G,\C^{m\times n})\cap\Per(\G,\C^{m\times n})=\{0\}$ and thus, there is no ambiguity.
\end{Remark}
\begin{Definition}\label{Definition:Convolution}
For all $u\in L^1(\G,\C^{l\times m})$ and $v\in\Per(\G,\C^{m\times n})$ we define the convolution $u*v\in\Per(\G,\C^{l\times n})$ by
\[u*v(g):=\sum_{h\in\G}u(h)v(h^{-1}g)\qquad\text{for all }g\in\G.\]
\end{Definition}
\begin{Lemma}\label{Lemma:Convolution}
Let $u\in L^1(\G,\C^{l\times m})$, $v\in\Per(\G,\C^{m\times n})$ and $\rho$ be a periodic representation of $\G$.
Then
\begin{enumerate}
\item the convolution $u*v$ is $\T^N$-periodic if $v$ is $\T^N$-periodic and
\item we have
\[\fourier{u*v}(\rho)=\fourier u(\rho)\fourier v(\rho).\]
\end{enumerate}
\end{Lemma}
\begin{proof}
Let $u\in L^1(\G,\C^{l\times m})$, $v\in\Per(\G,\C^{m\times n})$ and $\rho$ be a periodic representation of $\G$.
Let $N\in \MM$ such that $v$ and $\rho$ are $\T^N$-periodic.
By Definition~\ref{Definition:Convolution} it is clear that $u*v$ is $\T^N$-periodic and thus we have $u*v\in\Per(\G,\C^{m\times n})$ as claimed in Definition~\ref{Definition:Convolution}.
We have
\begin{align*}
\fourier{u*v}(\rho)&=\frac1{\abs{\CC_N}}\sum_{g\in\CC_N}u*v(g)\otimes\rho(g)\\
&=\frac1{\abs{\CC_N}}\sum_{g\in\CC_N}\sum_{h\in\G}\parens[\big]{u(h)v(h^{-1}g)}\otimes\rho(g)\\
&=\frac1{\abs{\CC_N}}\sum_{g\in\CC_N}\sum_{h\in\G}\parens[\big]{u(h)\otimes\rho(h)}\parens[\big]{v(h^{-1}g)\otimes\rho(h^{-1}g)}\\
&=\parens[\bigg]{\sum_{h\in\G}u(h)\otimes\rho(h)}\parens[\bigg]{\frac1{\abs{\CC_N}}\sum_{g\in\CC_N}v(g)\otimes\rho(g)}\\
&=\fourier u(\rho)\fourier v(\rho).\qedhere
\end{align*}
\end{proof}
%

%-------------------------------------------------------------------
%-------------------------------------------------------------------
\section{Quotient groups as semidirect products}\label{Subsection:SemidirectProduct}
By Definition~\ref{Definition:SubsetN} for all $m\in \MM$ the group $\T^m$ is a normal subgroup of $\G$, but in general there does not exist any group $\H<\G$ such that $\G=\T^m\rtimes\H$, see Example~\ref{Example:symmorphicnonspaceNEW}.
In this section we determine for $m\in \MM$ and appropriate $N\in m\N$ a group $\H\le\G/\T^N$ such that
\[\G/\T^N=\T^m/\T^N\rtimes \H,\]
see Theorem~\ref{Theorem:semidirectproductG_NNEW}.
The proof is similar to the proof of the Schur-Zassenhaus theorem, see, \eg, \cite{Alperin1995}.
If $\G$ is a space group, for appropriate $N\in\N$ the existence of a group $\H$ such that $\G/\T^N=\T/T^N\rtimes\H$ is mentioned in \cite[p.\,299]{Bacry1988} and in \cite[p.\,376]{Florek1993}.
%in the article 'Finite Space Groups Revisited' from R. Dirl and B.L. Davies on page 376.
\begin{Example}[Symmorphic and nonsymmorphic space groups]\label{Example:symmorphicnonspaceNEW}
Let $\G$ be a space group and $\T$  its subgroup of translations.
If there exists a group $\H<\G$ such that $\G=\T\rtimes\H$, then $\G$ is said to be a \emph{symmorphic} space group, see \eg, \cite[Section 9.1]{Opechowski1986}.
Otherwise, $\G$ is a \emph{nonsymmorphic} space group.

Let $d=2$, $t_1=(I_2,e_1)$, $t_2=(I_2,e_2)$, $\id=\iso{I_2}0$, $p_1=\iso[\big]{\parens[\big]{\begin{smallmatrix}1& 0 \\ 0& -1\end{smallmatrix}}}0$ and $p_2=\iso[\big]{\parens[\big]{\begin{smallmatrix}1& 0 \\ 0& -1\end{smallmatrix}}}{(\begin{smallmatrix}0.5 \\ 0 \end{smallmatrix})}$.
The space group
\[\set[\big]{tp}{t\in\angles{t_1,t_2},p\in\braces{\id,p_1}}<\E(2)\]
is symmorphic and equal to $\T\rtimes\H$ with $\T=\angles{t_1,t_2}$ and $\H=\angles{p_1}$.
The space group
\[\set[\big]{tp}{t\in\angles{t_1,t_2},p\in\braces{\id,p_2}}<\E(2)\]
is nonsymmorphic, since it does not contain any element of order 2, but the order of the quotient group of the space group by its subgroup of all translations is 2.
\end{Example}
\begin{Definition}
Let $\tildetau\colon\rot(\SG)\to\trans(\SG)$ be a map such that $\iso P{\tildetau(P)}\in\SG$ for all $P\in\rot(\SG)$.
We define the map
\begin{align*}
\bartau\colon\rot(\SG)\times\rot(\SG)\to\trans(\T_\SG), \qquad
(P,Q)\mapsto \tildetau(P)+P\tildetau(Q)-\tildetau(PQ).
\end{align*}
Furthermore, for all $n\in\N$ coprime to $\abs{\rot(\SG)}$ we define the set
\[\P_\SG^{(n)}:=\set[\bigg]{\iso[\bigg]{P}{\tildetau(P)-a(n)\sum_{Q\in\rot(\SG)} \bartau(P,Q)}}{P\in\rot(\SG)}\subset\SG,\]
where $a(n)=\max\set[\big]{\tilde a\in\braces{0,-1,\dots}}{\exists\,b\in\Z\text{ such that }\tilde a\abs{\rot(\SG)}+bn=1}$.

For all $n\in\N$ coprime to $\abs{\rot(\SG)}$ let $\P^{(n)}\subset\G$ be such that the map
\begin{align*}
\P^{(n)}\to\P_\SG^{(n)}, \qquad
g\mapsto\pi(g)
\end{align*}
is bijective, where $\pi\colon\G\to\SG$ is the natural surjective homomorphism.
\end{Definition}
\begin{Remark}
For all $P,Q\in\rot(\SG)$ it holds
\[\iso{P}{\tildetau(P)}\iso{Q}{\tildetau(Q)}=\iso{I_{d_2}}{\bartau(P,Q)}\iso{PQ}{\tildetau(PQ)}\]
and thus, the map $\bartau$ is well-defined.
\end{Remark}
If $n=1$, then $a(n)=0$ and $\P_\SG^{(n)}=\set{\iso P{\tildetau(P)}}{P\in\rot(\SG)}$.
\begin{Lemma}\label{Lemma:coprime}
For all $n\in\N$ coprime to $\abs{\rot(\SG)}$ and for all $N\in(n\N)\cap \MM$ it holds
\[\T^n\F\P^{(n)}\le\G\quad\text{and}\quad\T^N\triangleleft\T^n\F\P^{(n)}.\]
\end{Lemma}
\begin{proof}
Let $n\in\N$ be coprime to $\abs{\rot(\SG)}$.

First, we prove that $\T_\SG^n\P_\SG^{(n)}$ is a subgroup of $\SG$.
Let $t,s\in\T_\SG^n$ and $p,q\in\P_\SG^{(n)}$.
We have to show that $tp(sq)^{-1}\in\T_\SG^n\P_\SG^{(n)}$.
Clearly, it holds $tp(sq)^{-1}=tpq^{-1}s^{-1}(pq^{-1})^{-1}pq^{-1}$.
Since $\T_\SG^n\triangleleft\SG$, we have $(pq^{-1})s^{-1}(pq^{-1})^{-1}\in\T_\SG^n$, and hence, it suffices to show that $pq^{-1}\in\T_\SG^n\P_\SG^{(n)}$.
Let $P=\rot(p)$, $Q=\rot(q)$ and $R=PQ^{-1}\in\rot(\SG)$.
Let $a=\max\set[\big]{\tilde a\in\braces{0,-1,\dots}}{\exists\,b\in\Z\text{ such that }a\abs{\rot(\SG)}+bn=1}$ and $b\in\Z$ such that $a\abs{\rot(\SG)}+bn=1$.
We compute
\begin{align*}
pq^{-1}&=\iso[\bigg]{P}{\tildetau(P)-a\sum_{S\in\rot(\SG)}\bartau(P,S)}\iso[\bigg]{Q^{-1}}{-Q^{-1}\tildetau(Q)+a\sum_{S\in\rot(\SG)}Q^{-1}\bartau(Q,S)}\\
&=\iso[\bigg]{R}{\tildetau(P)-PQ^{-1}\tildetau(Q)-a\sum_{S\in\rot(\SG)}(\bartau(P,S)-PQ^{-1}\bartau(Q,S))}\\
&=\iso[\bigg]{R}{\tildetau(R)-\bartau(PQ^{-1},Q)-a\sum_{S\in\rot(\SG)}(\bartau(P,S)-PQ^{-1}\bartau(Q,S))}\\
&=\iso[\bigg]{R}{\tildetau(R)-(a\abs{\rot(\SG)}+bn)\bartau(PQ^{-1},Q)-a\sum_{S\in\rot(\SG)}\parens{\bartau(P,S)-PQ^{-1}\bartau(Q,S)}}\\
&=\iso[\Big]{I_{d_2}}{\bartau(R,Q)}^{-bn}\iso[\bigg]{R}{\tildetau(R)-a\sum_{S\in\rot(\SG)}\parens{\bartau(PQ^{-1},Q)+\bartau(P,S)-PQ^{-1}\bartau(Q,S)}}\\
&=\iso[\Big]{I_{d_2}}{\bartau(R,Q)}^{-bn}\iso[\bigg]{R}{\tildetau(R)-a\sum_{S\in\rot(\SG)}\parens{\tildetau(PQ^{-1})-\tildetau(PS)+PQ^{-1}\tildetau(QS)}}.
\intertext{We use that $\sum_{S\in\rot(\SG)}\tildetau(S)=\sum_{S\in\rot(\SG)}\tildetau(TS)$ for all $T\in\rot(\SG)$.}
pq^{-1}&=\iso[\Big]{I_{d_2}}{\bartau(R,Q)}^{-bn}\iso[\bigg]{R}{\tildetau(R)-a\sum_{S\in\rot(\SG)}\parens{\tildetau(PQ^{-1})-\tildetau(PQ^{-1}S)+PQ^{-1}\tildetau(S)}}\\
&=\iso[\Big]{I_{d_2}}{\bartau(R,Q)}^{-bn}\iso[\bigg]{R}{\tildetau(R)-a\sum_{S\in\rot(\SG)}\bartau(R,S)}\in \T_\SG^n\P_\SG^{(n)}.
\end{align*}
Thus, we have $\T_\SG^n\P_\SG^{(n)}\le\SG$.

Let $\pi$ be the natural surjective homomorphism from $\G$ to $\SG$ with kernel $\F$.
It holds $\pi^{-1}(\T_\SG^n\P_\SG^{(n)})=\T^n\F\P^{(n)}$ and thus, $\T^n\F\P^{(n)}$ is a subgroup of $\G$.

Now let $N\in(n\N)\cap \MM$.
Since $n$ divides $N$, we have $\T^N\subset\T^n\F\P^{(n)}$.
Since $N\in \MM$, we have $\T^N\triangleleft\T^n\F\P^{(n)}$.
\end{proof}
Recall Definition~\ref{Definition:QuotientGroup}.
\begin{Remark}\label{Remark:TFPn}
Let $n\in\N$ be coprime to $\abs{\rot(\SG)}$.
Let $m\in \MM$, $N=nm$ and $t_1,\dots,t_{d_2}\in\T^n$ such that $\pi(\{t_1,\dots,t_{d_2}\})$ generates $\T_\SG^n$, where $\pi\colon\T\F\to\T_\SG$ is the natural surjective homomorphism.
Then, the map
\begin{align*}
&\{0,\dots,m-1\}^{d_2}\times\F\times\P^{(n)}\to(\T^n\F\P^{(n)})_N,&&((n_1,\dots,n_{d_2}),f,p)\mapsto t_1^{n_1}\dots t_{d_2}^{n_{d_2}}fp\T^N
\end{align*}
is bijective.
\end{Remark}
The following lemma characterizes the elements of the groups $\G_N$, $(\T^n\F)_N$ and $(\T^m)_N$ for appropriate $n,m,N\in\N$.
\begin{Lemma}\label{Lemma:ElementsTFPNEW}
Let $t_1,\dots,t_{d_2}\in\T$ such that the set $\pi(\{t_1,\dots,t_{d_2}\})$ generates $\T_\SG$, where $\pi\colon\T\F\to\T_\SG$ is the natural surjective homomorphism.
For all $N\in \MM$ it holds
\begin{align*}
&\G_N=\set[\Big]{t_1^{n_1}\dots t_{d_2}^{n_{d_2}}fp\T^N}{n_1,\dots,n_{d_2}\in\{0,\dots,N-1\},f\in\F,p\in\P^{(1)}}
\intertext{and particularly $\abs{\G_N}=N^{d_2}\abs\F\abs{\rot(\SG)}$.\newline
For all $n\in\N$ and $N\in(n\N)\cap \MM$ it holds}
&(\T^n\F)_N=\set[\Big]{t_1^{nn_1}\dots t_{d_2}^{nn_{d_2}}f\T^N}{n_1,\dots,n_{d_2}\in\{0,\dots,(N/n)-1\},f\in\F}
\intertext{and particularly $\abs{(\T^n\F)_N}=(N/n)^{d_2}\abs\F$.
Moreover, for all $n\in\N$ and $N\in(n\N)\cap \MM$ it holds $(\T^n\F)_N\triangleleft\G_N$.\newline
For all $m\in \MM$ and $N\in m\N$ it holds}
&(\T^m)_N=\set[\Big]{t_1^{mn_1}\dots t_{d_2}^{mn_{d_2}}\T^N}{n_1,\dots,n_{d_2}\in\{0,\dots,(N/m)-1\}},
\end{align*}
$(\T^m)_N$ is a subgroup of the center of $(\T\F)_N$ and particularly $\abs{(\T^m)_N}=(N/m)^{d_2}$.
\end{Lemma}
\begin{proof}
Since $\P^{(1)}$ is a representation set of $\G/\T\F$, the map $\T\times\F\times\P^{(1)}\to\G$, $(t,f,p)\mapsto tfp$ is bijective.
The assertions are clear by Lemma~\ref{Lemma:TFP}, Theorem~\ref{Theorem:Ring} and Lemma~\ref{Lemma:helptranslationgroup}, Proposition~\ref{Proposition:center2}.
\end{proof}
The following theorem characterizes the group $\G_N$ for appropriate $N\in\N$.
\begin{Theorem}\label{Theorem:semidirectproductG_NNEW}
Let $m\in \MM$.
Let $n\in \N$ be coprime to $m$ and $\abs{\rot(\SG)}$.
Let $N=nm$.
Then, we have
\[\G_N=(\T^m)_N\rtimes(\T^n\F\P^{(n)})_N\]
and $(\T^m)_N$ is isomorphic to $\Z_n^{d_2}$.
\end{Theorem}
\begin{proof}
Let $m\in \MM$.
Let $n\in\N$ be coprime to $m$ and $\abs{\rot(\SG)}$.
Let $N=nm$.
By Theorem~\ref{Theorem:Ring} we have $\T^m\triangleleft\G$ and $\T^N\triangleleft\G$, and by Lemma~\ref{Lemma:helptranslationgroup} we have $\T^N\triangleleft\T^m$.
Hence, we have
%by the third isomorphism theorem 
%
\begin{equation}\label{eq:E1NEW}
(\T^m)_N\triangleleft\G_N.
\end{equation}
By Lemma~\ref{Lemma:helptranslationgroup} the group $\T^m$ is isomorphic to $\Z^{d_2}$ and thus, $(\T^m)_N$ is isomorphic to $\Z_n^{d_2}$.
By Lemma~\ref{Lemma:coprime} we have
\begin{equation}\label{eq:E2NEW}
(\T^n\F\P^{(n)})_N\le\G_N.
\end{equation}
For all $N\in\N$ and $\H\le\SG$ such that $\T_\SG^N$ is a
subgroup of $\H$, we denote
\[\H_N:=\H/\T_\SG^N.\]
Let $\pi\colon\G_N\to\SG_N$ be the natural surjective homomorphism with kernel $\set{g\T^N}{g\in\F}$.
We have
\begin{align}\label{eq:zt2NEW}
\pi((\T^m)_N\cap(\T^n\F\P^{(n)})_N)&\subset\pi((\T^m)_N)\cap\pi((\T^n\F\P^{(n)})_N)\nonumber\\
&=(\T_\SG^m)_N\cap(\T_\SG^n\P_\SG^{(n)})_N\nonumber\\
&=(\T_\SG^m)_N\cap(\T_\SG^n)_N\nonumber\\
&=\{\id\},
\end{align}
where in the third step we used that for all $p\in\P_\SG^{(n)}$ such that $\rot(p)=I_{d_2}$ we have $p\in\T_\SG^n$ and in the last step we used that the numbers $n^{d_2}$ and $m^{d_2}$ are coprime, $\abs{(\T_\SG^m)_N}=n^{d_2}$, $\abs{(\T_\SG^n)_N}=m^{d_2}$ and Lagrange's theorem.
By \eqref{eq:zt2NEW} and since $\pi|_{(\T^m)_N}$ is injective, we have
\begin{equation}\label{eq:E3NEW}
(\T^m)_N\cap(\T^n\F\P^{(n)})_N=\{\id\}.
\end{equation}
We have
\begin{equation}\label{eq:aoneNEW}
\abs{\G_N}=\abs{\ker(\pi)}\abs{\pi(\G_N)}=\abs{\set{g\T^N}{g\in\F}}\abs{\SG_N}=\abs\F\abs{\rot(\SG)} N^{d_2},
\end{equation}
see Lemma~\ref{Lemma:ElementsTFPNEW}, and
\begin{align}\label{eq:athreeNEW}
\abs{(\T^n\F\P^{(n)})_N}&=\abs{\ker(\pi|_{(\T^n\F\P^{(n)})_N})}\abs{\pi((\T^n\F\P^{(n)})_N)}\nonumber\\
&=\abs\F\abs{(\T_\SG^n\P_\SG^{(n)})_N}=\abs\F\abs{\P_\SG^{(n)}}\abs{(\T_\SG^n)_N}=\abs\F\abs{\rot(\SG)} m^{d_2},
\end{align}
see Remark~\ref{Remark:TFPn}.
By \eqref{eq:aoneNEW}, \eqref{eq:athreeNEW} and since $(\T^m)_N$ is isomorphic to $\Z_n^{d_2}$, we have
\begin{equation}\label{eq:E5NEW}
\abs{\G_N}=\abs{(\T^m)_N}\abs{(\T^n\F\P^{(n)})_N}.
\end{equation}
By \eqref{eq:E1NEW}, \eqref{eq:E2NEW}, \eqref{eq:E3NEW} and \eqref{eq:E5NEW} we have
\[\G_N=(\T^m)_N\rtimes (\T^n\F\P^{(n)})_N.\qedhere\]
\end{proof}
\begin{Corollary}\label{Corollary:semidirectproductG_NNEW}
Let $m\in \MM$, $\tilde n\in \N$, $n=\tilde nm\abs{\rot(\SG)}+1$ and $N=nm$.
Then we have
\begin{align*}
&\P_\SG^{(n)}=\set[\bigg]{\iso[\bigg]{P}{\tildetau(P)+\tilde nm\sum_{Q\in\rot(\SG)} \bartau(P,Q)}}{P\in\rot(\SG)}
\shortintertext{and}
&\G_N=(\T^m)_N\rtimes(\T^{n}\F\P^{(n)})_N.
\end{align*}
\end{Corollary}
\begin{proof}
Let $m\in \MM$, $\tilde n\in \N$, $n=\tilde nm\abs{\rot(\SG)}+1$ and $N=nm$.
In particular, $n$ is coprime to $m$ and $\abs{\rot(\SG)}$.
We have
\begin{align*}
&\max\set[\big]{\tilde a\in\braces{0,-1,\dots}}{\exists\,b\in\Z\text{ such that }\tilde a\abs{\rot(\SG)}+bn=1}\\
&\qquad=\max\set[\big]{\tilde a\in\braces{0,-1,\dots}}{\exists\,b\in\N\text{ such that }(\tilde a+b\tilde nm)\abs{\rot(\SG)}+b=1}\\
&\qquad=-\tilde nm
\end{align*}
and hence,
\[\P_\SG^{(n)}=\set[\bigg]{\iso[\bigg]{P}{\tildetau(P)+\tilde nm\sum_{Q\in\rot(\SG)} \bartau(P,Q)}}{ P\in\rot(\SG)}.\]
By Theorem~\ref{Theorem:semidirectproductG_NNEW} we have $\G_N=(\T^m)_N\rtimes(\T^{n}\F\P^{(n)})_N$.
\end{proof}
\begin{Corollary}\label{Corollary:helpsemidirectproductNEW}
Suppose that $\G$ is a space group.
Let $N\in\N$ be coprime to $\abs{\rot(\G)}$. Then we have
\[\G_N=\T_N\rtimes\set{g\T^N}{g\in\P^{(N)}}.\]
\end{Corollary}
\begin{proof}
Let $\G$ be a space group.
We have $\F=\{\id\}$ and $\MM=\N$.
For all $N\in\N$ coprime to $\abs{\rot(\G)}$, we have $(\T^N\P^{(N)})/\T^N=\set{g\T^N}{g\in\P^{(N)}}$.
Thus, Theorem~\ref{Theorem:semidirectproductG_NNEW} implies the assertion.
\end{proof}
\begin{Corollary}
Suppose that $\G$ is a space group.
Let $n\in\N$ and $N=n\abs{\rot(\G)}+1$.
Then it holds
\begin{align*}
&\P^{(N)}=\set[\bigg]{\iso[\bigg]{P}{\tildetau(P)+n\sum_{Q\in\rot(\G)} \bartau(P,Q)}}{ P\in\rot(\G)}
\shortintertext{and}
&\G_N=\T_N\rtimes \set{g\T^N}{g\in\P^{(N)}}.
\end{align*}
\end{Corollary}
\begin{proof}
This is clear by Corollary~\ref{Corollary:semidirectproductG_NNEW} and Corollary~\ref{Corollary:helpsemidirectproductNEW}.
\end{proof}
\begin{Corollary}
Suppose that $\G=\T\F$.
Let $m\in \MM$ and $n\in\N$ be coprime.
Let $N=nm$.
Then it holds
\[\G_N=(\T^m)_N\times(\T^n\F)_N.\]
\end{Corollary}
\begin{proof}
Suppose that $\G=\T\F$.
Let $m\in \MM$ and $n\in\N$ be coprime.
We have $\SG=\T_\SG$ and $\rot(\SG)=\{I_{d_2}\}$.
Without loss of generality we assume that $\tildetau=0$.
We have $\bartau=0$ and $\P_\SG^{(n)}=\{\id\}$.
Without loss of generality we assume that $\P^{(n)}=\{\id\}$.
By Theorem~\ref{Theorem:semidirectproductG_NNEW}, Lemma~\ref{Lemma:TFP}\ref{item:abc} and Proposition~\ref{Proposition:center2} we have $\G_N=(\T^m)_N\times(\T^n\F)_N$.
\end{proof}
%

%-------------------------------------------------------------------
%-------------------------------------------------------------------
%\appendix
\begin{appendices}
\section{Dual spaces and induced representations}\label{subsection:harmonic}
With the aim to keep our presentation largely self-contained and to provide easy references, we recall some topics in representation theory in a short appendix. In our set-up it is not restrictive to only consider finite-dimensional representations, see Remark~\ref{Remark:DefRep} below.
\begin{Definition}\label{Definition:Representation}
Let $\H$ be a finite group or a discrete subgroup of $\E(d)$.
A \emph{representation} of $\H$ is a homomorphism $\rho\colon\H\to\U(d_\rho)$, where $d_\rho\in\N$ is the \emph{dimension} of $\rho$ and $\U(d_\rho)$ is the group of all unitary matrices in $\C^{d_\rho\times d_\rho}$. 
Two representations $\rho,\rho'$ of $\H$ are said to be \emph{(unitarily) equivalent} if $d_\rho=d_{\rho'}$ and there exists some $T\in\U(d_\rho)$ such that
\[T^H\rho(g)T=\rho'(g)\qquad\text{for all }g\in\H.\]
A representation $\rho$ of $\H$ is said to be \emph{irreducible} if the only subspaces of $\C^{d_\rho}$ invariant under $\set{\rho(g)}{g\in\H}$ are $\{0\}$ and $\C^{d_\rho}$.
Let $\dual{\H}$ denote the set of all equivalence classes of irreducible representations of $\H$.
One calls $\dual{\H}$ the \emph{dual space} of $\H$.
If $\NN$ is a normal subgroup of $\H$, then the group $\H$ acts on the set of all representations of $\NN$ by
\[g\cdot \rho(n):=\rho(g^{-1}n g)\qquad\text{for all }g\in\H\text{, representations }\rho\text{ of }\NN \text{ and }n\in\NN.\]
For given representations $\rho_1,\dots,\rho_n$ of $\H$, we define the \emph{direct sum}
\begin{align*}
\oplus_{i=1}^n\rho_i\colon&\H\to\U(d), \quad 
g\mapsto\oplus_{i=1}^n(\rho_i(g)),
\end{align*}
where $d=\sum_{i=1}^nd_{\rho_i}$.
In a canonical way, the above group action and terms \emph{dimension}, \emph{irreducible} and \emph{direct sum} are also defined for equivalence classes of representations. 
For any $\rho \in \dual\NN$ we denote by $\H^\rho$ the stabilizer of $\rho$ in $\H$.
\end{Definition}
\begin{Remark}\label{Remark:DefRep}
In \cite{Moore1972} the following theorem is proved for any locally compact group:
There exists an integer $M\in\N$ such that the dimension of every irreducible representation is less than or equal to $M$ if and only if there is an open abelian subgroup of finite index.
This, in particular, applies to finite groups and discrete subgroups of $\E(d)$.
\end{Remark}
A caveat on notation:
For a representation and for an equivalence class of representations we use the symbol $\chi$ if it is one-dimensional and $\rho$ otherwise.
For every one-dimensional representation $\chi$ its equivalence class is a singleton which we also call a \emph{representation} and denote by $\chi$.

It is elementary to see that, if $\H<\E(d)$ is a discrete subgroup of $\E(d)$, then characters (one-dimensional representations) act by multiplication on $\dual\H$. 
Moreover, for any normal subgroup $\NN$ of $\H$ the action of $\NN$ on $\dual\NN$ is trivial: 
\begin{equation}\label{eq:ActionTrivial}
g\cdot\rho=\rho\qquad\text{for all }g\in\NN\text{ and }\rho\in\dual\NN, 
\end{equation}
\ie, $\H^\rho \supset \NN$. As a consequence, also $\H/\NN$ acts on $\dual \NN$. For $\rho \in \dual\NN$ we denote by $(\H/\NN)^\rho$ the stabilizer and by $(\H/\NN) \cdot \rho$ the orbit of $\rho$ with respect to  this action. 
Also the following two propositions are well-known. Their proofs (in the more general setting of locally compact groups) can be found, \eg, in \cite{Kaniuth2013}.% (See Proposition 1.35, respectively, Proposition 1.71.) 
\begin{Proposition}\label{Proposition:commutant}
Let $\rho\colon\H\to\U(d_\rho)$ be a representation of a discrete group $\H<\E(d)$.
Then $\rho$ is irreducible if and only if
\[\text{commutant of }\rho(\H):=\set{T\in\C^{d_\rho\times d_\rho}}{T\rho(g)=\rho(g)T\text{ for all }g\in \H}=\C I_{d_\rho}.
\]
\end{Proposition}
\begin{Proposition}\label{Proposition:homeomorphism}
Let $\NN$ be a normal subgroup of a discrete group $\H<\E(d)$ and $\pi\colon \H\to \H/\NN$ be the quotient homomorphism.
The map $\rho\mapsto\rho\circ \pi$ is a homeomorphism of $\dual{\H/\NN}$ with the subset of $\dual\H$ consisting of those elements of $\dual\H$ which annihilate $\NN$.
\end{Proposition}
Our setting allows for a definition of induced representations as for finite groups, see, \eg, \cite[Section 8.2]{Steinberg2012}. (For general locally compact groups the definition is more complicated, see, \eg, \cite[Chapter 2]{Kaniuth2013}.)
\begin{Definition}\label{Definition:IndRep}
Let $\H<\E(d)$ be discrete and $\K$ be a subgroup of $\H$ such that the index $n=\abs{\H:\K}$ is finite.
Choose a complete set of representatives $\{h_1,\dots,h_n\}$ of the left cosets of $\K$ in $\H$.
Suppose $\rho\colon\K\to\U(d_\rho)$ is a representation of $\K$.
Let us introduce a dot notation in this context by setting
\begin{align*}
\dot\rho(g):=\begin{cases}
      \rho(g) & \text{if }g\in\K\\
      0_{d_\rho,d_\rho} & \text{else}
    \end{cases}
\end{align*}
for all $g\in\H$.
The \emph{induced representation} $\Ind_\K^\H\rho\colon\H\to\U(nd_\rho)$ is defined by
\[\Ind_\K^\H\rho(g)=\begin{bmatrix}\dot\rho(h_1^{-1}gh_1)&\cdots&\dot\rho(h_1^{-1}gh_n)\\ \vdots&\ddots&\vdots\\\dot\rho(h_n^{-1}gh_1)&\cdots&\dot\rho(h_n^{-1}gh_n)\end{bmatrix}\qquad\text{for all }g\in\H.\]
The \emph{induced representation} of an equivalence class of representations is the equivalence class of the induced representation of a representative. 
Moreover, let $\Ind_\K^\H(\dual\K)$ denote the set of all induced representations of $\dual\K$.
We also write $\Ind$ instead of $\Ind_\K^\H$ if $\K$ and $\H$ are clear by context.
\end{Definition}
The following proposition is standard in Clifford theory. %cf.\ \cite[Proposition 2.39]{Kaniuth2013}. Für die Injektivität: Theorem 3.2(ii) of ON INDUCED REPRESENTATIONS OF DISCRETE GROUPS MICHAEL W. BINDER. 
\begin{Proposition}\label{Proposition:Induced}
Let $\H<\E(d)$ be discrete and $\NN$ be a normal subgroup of $\H$ such that the index $\abs{\H:\NN}$ is finite.
Then the map
\begin{align*}
\dual\NN/\H\to\Ind_\NN^\H(\dual\NN), \qquad 
\H\cdot\rho\mapsto\Ind_\NN^\H\rho
\end{align*}
is bijective, where $\dual\NN/\H=\set{\H\cdot\rho}{\rho\in\dual\NN}$.
\end{Proposition}
\begin{proof}
That $\varphi$ is well-defined and surjective follows, \eg, from \cite[Proposition 2.39]{Kaniuth2013}. For the injectivity of $\varphi$ see, \eg, \cite[Theorem 3.2(ii)]{Binder1993}. 
\end{proof}
We also record the following well-known results on the relation of irreducible and induced representations adapted to our setting. %
\begin{Proposition}\label{Proposition:Teildarstellung}
Let $\H<\E(d)$ be discrete and $\K$ be a subgroup of $\H$ of finite index. Then every irreducible representation of $\H$ is a subrepresentation of a representation which is induced by an irreducible representation of $\H$. 
\end{Proposition}
\begin{proof}
Let $\rho$ be an irreducible representation of $\H$. Choose $\sigma\in\dual \K$ such that $\sigma$ is a subrepresentation of $\rho$ restricted to $\H$. Then \cite[Theorem 8.2]{Mackey:52} shows that $\rho$ is a subrepresentation of $\Ind_{\K}^{\H}\sigma$.
\end{proof}
The following is Mackey's irreducibility criterion, cp.\ \cite{Mackey:51}, which we state in a form that is directly implied by, \eg, \cite[Theorem 1.1]{Bekka2003}. 
\begin{Theorem}\label{Theorem:Mackey}
Let $\H<\E(d)$ be discrete and $\NN$ be a normal subgroup of $\H$ of finite index. Let $\rho$ be an irreducible representation of $\NN$. Then $\Ind_{\NN}^{\H}\rho$ is irreducible if and only if for every $g\in \H \setminus \NN$ the representations $g\cdot \rho$ and $\rho$ are not isomorphic.
\end{Theorem}
The Haar measure (\ie, the counting measure) on a discrete group of Euclidean isometries induces a measure on the dual space as follows, cf., \eg,  \cite{Dixmier1977}. 
\begin{Theorem}\label{Theorem:Plancherel-measure}
For any discrete subgroup $\H$ of $\E(d)$ there is a unique measure $\mu_{\H}$ on $\dual\H$, the so-called \emph{Plancherel measure}, such that 
\[ \sum_{g \in \H} |f(g)|^2 
    = \int_{\dual \H} \operatorname{Tr} \big( \rho(f)^H \rho(f) \big) \, \mathrm{d}\mu_{\H}, 
\]
for all $f \in \ell^1(\H)$, where $\rho(f) = \sum_{g \in \H} f(g) \rho(g) \in \C^{d_\rho\times d_\rho}$. 
\end{Theorem}
(If $\H$ is abelian, then $\mu_\H$ is (a multiple of) the Haar measure of $\dual \H$.) We finally state a version of the Mackey machine (cf.\ \cite{Mackey:58}) in a form proven in \cite{KleppnerLipsman:72} and adapted to our setting of discrete subgroups of $\E(d)$. 

\begin{Theorem}\label{Theorem:Mackey-machine}
Let $\H<\E(d)$ be discrete, $\NN$ be a normal abelian subgroup of $\H$ of finite index and set $\P = \H/\NN$. Suppose there is an $\H$-invariant measurable $\mu_\NN$-conull subset $\Omega\subset\dual \NN$ such that all $\sigma\in\Omega$ can be extended to a unitary representation $\tilde\sigma$ of $\H^\sigma$. For $\sigma \in \Omega$ and $\rho \in \P^\sigma$ define the unitary representation $\sigma\times\rho$ of $\H^\sigma$, acting on $\C^{d_\sigma d_\rho}$, by
\[(\sigma\times\rho)(g)=\tilde\sigma(g)\otimes\rho(g\NN)\qquad\text{for all }g\in\H^\sigma.\]
Then the set
\[\bigcup_{\mathcal \P \cdot \sigma\in\Omega/\P}\set[\big]{\Ind_{\H^\sigma}^\H(\sigma\times\rho)}{\rho\in\dual{\P^\sigma}}\]
is a $\mu_\H$-conull subset of $\dual \H$.
\end{Theorem}
\end{appendices}

\bibliography{Literature/Literature}

\end{document}